\documentclass[final]{siamltex}
\usepackage{animate}
\usepackage{amsmath}
\usepackage{amssymb}
\usepackage{graphics}
\usepackage{graphicx}
\usepackage{textcomp}
\usepackage{mathrsfs}
\usepackage{epstopdf}
\usepackage{array}
\usepackage{cite}
\usepackage[maxfloats=99]{morefloats}
\usepackage{color}
\usepackage{url}
\usepackage{color}
\usepackage{cite}
\usepackage{subcaption}

\usepackage{stmaryrd}
\usepackage{bm}
\usepackage{multirow}

\newtheorem{coro}[theorem]{Corollary}

\newtheorem{lmm}{Lemma}[section]
\newtheorem{thm}{Theorem}[section]

\renewcommand{\theequation}{\arabic{section}.\arabic{equation}}

\graphicspath{{fig/}}

\newtheorem{remark}{Remark}[section]

\title{Two types of spectral volume methods for 1-D linear hyperbolic equations with degenerate variable coefficients}

\author{ Minqiang Xu\thanks{College of Science, Zhejiang University of Technology, Hangzhou, 310023, P.R. China, and School of Computer Science and Engineering, Sun Yat-sen University, Guangzhou 510275, P.R. China. Email: xumq9@mail2.sysu.edu.cn. The research of this author was supported by the Opening Project of Guangdong Province Key Laboratory of Computational Science at the Sun Yat-sen University(2021008).}
\and Yanting Yuan\thanks{School of Computer Science and Engineering, Sun Yat-sen University, Guangzhou 510275, P.R. China. Email: yuanyt6@mail2.sysu.edu.cn}
\and Waixiang Cao\thanks{School of Mathematical Sciences,  Beijing Normal University, Beijing 100875, China.
Email: caowx@bnu.edu.cn. The research of this author was supported in part by NSFC Grant 11871106.}
\and Qingsong Zou\thanks{Corresponding author.
School of Computer Science and Engineering, and Guangdong Province Key
Laboratory of Computational Science, Sun Yat-sen University, Guangzhou 510275, China.
Email: mcszqs@mail.sysu.edu.cn. The research of this author was supported in part by NSFC Grant 12071496,
Guangdong Provincial NSF Grant 2017B030311001, and Guangdong Province Key Laboratory of Computational
Science at the Sun Yat-sen University(2020B1212060032)}}
\begin{document}

\maketitle

%
%
\medskip

\begin{abstract}
 In this paper, we analyze  two classes of spectral volume (SV) methods
 for one-dimensional  hyperbolic equations with degenerate variable coefficients.
 The two classes of SV methods are constructed by letting a piecewise $k$-th order ($k\ge 1$ is an arbitrary integer)
 polynomial function satisfy the local conservation law in each {\it control volume} obtained by
  dividing the interval element of the underlying mesh with $k$ Gauss-Legendre points (LSV) or  Radaus points (RSV).  The $L^2$-norm
  stability and optimal order convergence properties for both methods are rigorously proved for
 general non-uniform meshes. The superconvergence behaviors of the two SV schemes have been also investigated : it is proved that under the $L^2$ norm,
 the SV flux function approximates the exact flux with $(k+2)$-th order and  the SV solution approximates the exact solution with $(k+\frac32)$-th order; some superconvergence behaviors  at certain special points and for element averages have been also discovered and proved. Our theoretical findings are verified by several numerical experiments.

\vskip .3cm

\end{abstract}

\section{Introduction}
In this paper, we propose and analyze two classes of SV methods for solving the following one-dimensional linear hyperbolic equation:
\begin{eqnarray}\label{eqn:para}
\left \{
\begin{array}{lll}
u_t+(\alpha u)_x=g(x,t), \quad & \quad(x,t)\in[0,2\pi]\times [0,T],\\
~~~~u(x,0)=u_0(x), \quad & \quad x\in [0,2\pi],\\
\end{array}
\right.
\end{eqnarray}
where $u_0(x)$, $\alpha(x)$ and $g(x,t)$ are  all smooth. For simplicity,  here we only discuss the periodic boundary condition.
 We assume that $\alpha$ may degenerate and
 has a finite number of turning points on the whole domain $[0,2\pi]$.

Hyperbolic equations have wide  applications in chemical reactions, combustion, explosions, and multi-phase flow  problems,
transmission of electrical signal in the animal nervous system and so on. Due to the fact that the analytic solution are usually very difficult to be obtained,
the numerical simulation  becomes more and more important in the study of hyperbolic equations.
During the past several decades, numerical methods for hyperbolic  problems have been extensively studied. In particular, a lot of  high order (or high resolution)
numerical schemes have been designed and studied.
The  state-of-the-art high order schemes for hyperbolic equations include  the high-order k-exact  FV method \cite{FV1,FV2}, the essentially nonoscillatory (ENO) method \cite{ENO:0,ENO}, the weighted ENO (WENO) method
\cite{ENO:2,ENO:1}, and the discontinuous Galerkin (DG) method \cite{ReedHill1973,Cockburn;Shu:1989,Cockburn;Shu:JCP1998,Houston,Cockburn;Karniadakis},
and the spectral volume(SV)  mehtod   \cite{wang2002_1,wang2002_2,wang2004_1,wang2004_2,VanLacorWang2007 }, and so on.

 The SV method was introduced and studied by Z. J. Wang and his colleagues \cite{wang2002_1, wang2002_2,wang2004_1,wang2004_2,VanLacorWang2007,Liu-Vinokur-Wang,SunWangLiu2007, wang2006_2}.
 Similar to the well-known DG method, the SV method has many advantages such as : it
is of high-order, it can be established on nonuniform or unstructured grids, it has compact stencils, it only requires the information of the immediate cell neighbours to evaluate the residuals of one target cell and thus it can be easily parallelizable computing.  We also refer to  \cite{SunWang2014,zhang-shu}  for the
 comparisons between the SV and DG methods.

It is known that the  DG method has been heavily studied, including the theory for its stability and convergence (see,  e.g., \cite{Lesaint:1974,Richter,Cockburn,Johnson,Peterson}).
However, the theory for the SV method is far less developed. Different from the DG method, the stability of the SV method
is heavily dependent upon the suitable  partitions (i.e., the construction of the control volumes).
In \cite{s1,s2}, Van den Abeele et al. investigated the influence of the partitioning into control volumes and the wave propagation properties of the SV method for 1D and 2D hyperbolic equations. with the help of the wave  properties, they showed that the third and fourth-order  SV schemes based on the Gauss-CLobatto distributions are weakly unstable. Later, Van den Abeele et al. in \cite{VdA2009} adopt the matrix method to study the stability 
of second and third-order SV schemes for 3D tetrahedral grids, and proved that the second SV schemes was stable while a two-parameters family of third-order scheme proposed in the literature is unstable. Zhang and Shu in \cite{zhang-shu} use a method of   Fourier type analysis to study the stability of $p$-th order ($p\le 3$) SV schemes 
 applied to one  dimensional hyperbolic conservation laws for uniform meshes. 
   Very recently,  Cao and Zou in \cite{cao-zou:sv}  presented and studied two stable SV schemes of any order for one dimensional hyperbolic conservation laws,
    and proved that for a general nonuniform mesh,  the two SV schemes (i.e., LSV and RSV) are stable and can achieve optimal convergence orders in the $L^2$  norms. Moreover, the superconvergence properties of the SV solution at some special points as well as for the
 cell averages are also  investigated.  A surprising fact that the RSV method is identical to the upwind DG method for hyperbolic equations  with
 constant coefficients is also revealed, with a rigorous mathematical proof.

 The current work is the second in a series to study
  the stability, accuracy, convergence and
 superconvergence properties of the SV method, where the degenerate variable coefficients
  hyperbolic equation  \eqref{eqn:para} is under  concerned.  Compared with constant coefficients problems, the analysis 
  for problems with  degenerate variable coefficients is much more sophisticated: on the one hand, as the coefficient changes signs, 
  a suitable numerical flux as well as a partition into the control volumes are of great importance to ensure the stability of the SV methods; 
  on the other hand, the degenerate coefficients may lead to a descent of the convergence rate, which makes the convergence analysis 
  more difficult.  
  Note that superconvergence theory of the counterpart DG methods
has  extensively studied and  most of  the studies are based on the constant coefficients problems,  i.e., $\alpha=1$ in \eqref{eqn:para}
 (see, e.g., \cite{Adjerid;Massey2006,Adjerid;Weinhart2011,caozhangzou2014,Cao-Shu-ZhangM2AN,cao:shu:yang:zhang2Dg, Cao-Shu-Yang-Zhang:nonlinear,Chen;Shu:SIAM2010,Yang;Shu:SIAM2012}). In 2017, Cao et al. \cite{Cao-Shu-ZhangM2AN} studied the superconvergence of DG methods for the problem \eqref{eqn:para}
  and showed that   if the variable or nonlinear coefficients degenerate, i.e.,
  the coefficient changes signs or otherwise achieves the value zero,  the superconvergence results are of great differences from those for problems with constant coefficients
  or variable coefficients away from $0$.  To be more precise,
  the superconvergence rate may depend upon specific properties of the variable coefficient function $\alpha$ and the highest superconvergence rate is $h^{k+\frac 32}$,
 $\frac 12$ order higher than the optimal convergence rate.
 It is natural to ask whether the similar superconvergence phenomenon still exists for SV methods when applied to
hyperbolic problems with variable coefficients. 
In this work, we will provide a confirmatory answer to this question.
Inspired by the work in \cite{Cao-Shu-ZhangM2AN} and \cite{cao-zou:sv},  we first 
construct our control volume by using Gauss points or right/left Radau points of the underlying meshes, dependent on the sign 
of the variable coefficients, and then prove the specially designed any order SV schemes is $L^2$ stable for non-uniform meshes.
By defining a special projection of the exact solution and using the idea of correction idea, 
 we also study the convergence and 
  superconvergence behaviors of the SV method  for flux function approximation and solution approximation itself. 
  We show  that if the variable coefficients is away from $0$, then all the convergence results are the same as those for 
 problems with constant coefficients. While if the  coefficient changes its signs,  
 the highest superconvergence rate of the SV solution itself is $\mathcal{O}(h^{k+\frac{3}{2}})$, and for the SV flux function approximation,
  the superconvergence rate can be  improved to $\mathcal{O}(h^{k+2})$.  Finally,
  superconvergence  of the SV approximation  at some special points (which are identified as
  Radau points and Gauss points)  as well as for cell averages, are also obtained.

   To end this introduction,   we would like to emphasize that  the contribution of  the current study  lies in that:  On the one hand,
   we  present two SV schemes of any order  for  possibly degenerate variable coefficient problems  and prove the stability, convergence and
   superconvergence properties of the proposed numerical methods, fill the vacancy in the theory of the SV methods. On the other hand,
   the established theoretical findings for degenerate variable coefficient problems extend the   results for constant coefficient problems to the general case.
      By doing so, we enrich the  convergence  theory of the SV method for linear hyperbolic equations in one dimension.  Furthermore,
   the analysis technique used for linear hyperbolic equations with degenerate variable coefficients  is a necessary  and significant  step to
   study the nonlinear hyperbolic equations,  which arises in many applications such as computational fluid dynamics and computational electro-magnetism.



  The rest of the paper is organized as follows.
   In Section 2, we present two classes of SV methods
   for linear conservation laws with degenerate variable coefficients. In Section 3,  we study the stability of the RSV and LSV methods, where
  inequalities in energy norm and flux function norm are established.  In Section 4, optimal error estimates in the $L^2$ norm for both RSV and LSV are proved.
 Sections  5 is dedicated to the analysis of the superconvergence   behavior of the SV solution itself and flux function approximation.
 We show that the superconvergence phenomenon exists for general variable coefficient hyperbolic equations, and the superconvergence rate may depend upon specific properties of the variable coefficient function.
  In Section 6, we provide some numerical examples to support our theoretical findings.  Finally, some concluding remarks are presented in Section 7.

\section{RSV and LSV methods}

Let $\Omega=[0,2\pi]$ and $0=x_0<x_{\frac{1}{2}}<x_{\frac{3}{2}}<\cdots<x_{N+\frac{1}{2}}=2\pi$ be $N+1$ distinct points which partition the $\Omega$ into $N$ elements. For any positive integers $l$, we define $\mathbb{Z}_l=\{1,2,\cdots,l\}$ and $\mathbb{Z}_l^0=\{0,1,\cdots,l\}$. For any $i\in \mathbb{Z}_N$, we denote
\begin{equation*}
  \mathbf{V}_i=[x_{i-\frac{1}{2}},x_{i+\frac{1}{2}}],~\text{and}~h_i=x_{i+\frac{1}{2}}-x_{i-\frac{1}{2}}.
\end{equation*}
Let $\overline{h}_i=\frac{h_i}{2}$ and $h=\max \limits_{ j\in\mathbb{Z}_N }h_j$. We also assume that the mesh is shape-regular, i.e., $h \lesssim h_i, i\in \mathbb{Z}_N$. Let
$-1=s_0<s_1<s_2<\cdots<s_{k+1}=1$ be $k+2$ distinct  points in the reference element $[-1,1]$. Then we get a partition of each element $\mathbf{V}_i$ via the following linear transformation,
\begin{equation}\label{Subdividingpoints}
  x_{i,j}=\frac{h_i}{2}s_j+\frac{1}{2}(x_{i-\frac{1}{2}}+x_{i+\frac{1}{2}}),~i\in\mathbb{Z}_N,j\in \mathbb{Z}_{k+1}^0.
\end{equation}
Define the finite element space $V_h$ and the piecewise constant function space $\mathcal{V}_h$ by
\begin{equation*}
  V_h=\{v_h:v_h|_{\mathbf{V}_i}\in \mathcal{P}_k(\mathbf{V}_i),i\in \mathbb{Z}_N\},
\end{equation*}
\begin{equation*}
  \mathcal{V}_h=\{v_h^*:v_h^*|_{\mathbf{V}_{i,j}}\in \mathcal{P}_0(\mathbf{V}_{i,j}),i\in \mathbb{Z}_N,j\in \mathbb{Z}_k^0\},
\end{equation*}
where $\mathcal{P}_k$ is the space of polynomials of degree at most $k$.  
 Let
\begin{equation*}
  \mathcal{H}_h=\{v:v|_{\mathbf{V}_{i}}\in H^1,i\in \mathbb{Z}_N\}.
\end{equation*}
 Here and throughout this paper,  we  adopt standard notations for Sobolev spaces such as  $W^{m,p}(D)$ on sub-domain $D\subset\Omega$ equipped with the norm
$\|\cdot\|_{m,p,D}$ and semi-norm $|\cdot|_{m,p,D}$. When $D=\Omega$, we omit the index $D$; and if $p = 2$,
we set $W^{m,p}(D)=H^m(D)$ and $\|\cdot\|_{m,p,D}=\|\cdot\|_{m,D}$ and $|\cdot|_{m,p,D}=|\cdot |_{m,D}$.


 The SV scheme for \eqref{eqn:para} read as: Find $u_h\in V_h$ such that
\begin{equation}\label{SVschem}
  \int_{x_{i,j}}^{x_{i,j+1}}(u_h)_t\text{d}x+\alpha \widehat{u}_h|_{i,j+1}-\alpha \widehat{u}_h|_{i,j}=\int_{x_{i,j}}^{x_{i,j+1}}g\text{d}x,
\end{equation}
where $u|_{i,j}=u(x_{i,j})$ and $\hat{u}_h$ is the numerical flux. In this paper, we choose the upwind flux. That is,
\begin{eqnarray}\label{eqn:flux}
\widehat{u}_h|_{i+\frac{1}{2}}=\left \{
\begin{array}{ll}
u_h^{-}|_{i+\frac{1}{2}}, &\alpha|_{i+\frac{1}{2}}> 0,\\
u_h^{+}|_{i+\frac{1}{2}}, &\alpha|_{i+\frac{1}{2}}\leq 0,
\end{array}
\right.
\end{eqnarray}
where $u_h^{+}|_{i+\frac 12}$ and $u_h^{-}|_{i+\frac 12}$  denotes  the right and left limits of $u_h$ at the point $x_{i+\frac 12}$, respectively.

We observe that the choice of the partition points $\{x_{i,j}\}_{j=0}^{k+1}$ of $\mathbf{V}_{i}$ has a great influence on the SV scheme \eqref{SVschem}.
By  taking different $s_j$ or $x_{i,j}$, we get different SV schemes.
In this paper, we will consider two classes of SV schemes. One is constructed by using Radau points while the other   is  using Gauss-Legendre points.
We call the corresponding schemes as Radau spectral volume (RSV) method and Gauss Legendre spectral volume (LSV) method. To be more precise,

\emph{\textbf{LSV Scheme}}:  $\{s_j\}_{j=1}^k$ are chosen as Gauss-Legendre points, i.e., $s_j, j\in\mathbb{Z}_k$ are $k$ zeros of the Legendre polynomial $L_k$ of degree $k$.

\emph{\textbf{RSV Scheme}}:  $\{s_j\}_{j=1}^k$ are taken  according to the sign of $\alpha$ on the element boundary, i.e.,
\begin{itemize}
\item If $\alpha(x_{i-\frac{1}{2}})>0,\alpha(x_{i+\frac{1}{2}})>0$, $\{s_j\}_{j=1}^{k+1}$ are chosen as  right Radau points, i.e.,  zeros of the right Radau polynomial $L_{k+1}-L_k$.
\item If $\alpha(x_{i-\frac{1}{2}})<0,\alpha(x_{i+\frac{1}{2}})<0$, $\{s_j\}_{j=0}^{k}$ are chosen as  left Radau points, i.e.,  zeros of the right Radau polynomial $L_{k+1}+L_k$.
\item Otherwise, either of the left Radau points or right Radau points can be used.
\end{itemize}

We end with this section by discussing the above two SV schemes under the framework of a Petrov-Galerkin method.  Define  the SV bilinear form on $\mathcal{H}_h\times \mathcal{V}_h$ by
\begin{equation}\label{SV bilinear form on element}
  a_{h,i}(v,w^*)=\sum_{j=0}^kw_{i,j}^*\int_{x_{i,j}}^{x_{i,j+1}}v_t\text{d}x+\sum_{j=0}^kw_{i,j}^*\left(\alpha \widehat{v}|_{i,j+1}-\alpha \widehat{v}|_{i,j}\right),\ \
  \forall v\in \mathcal{H}_h,w^*\in\mathcal{V}_h
\end{equation}
here $w_{i,j}^*$ represents the value of $w^*$ on $\mathbf{V}_{i,j}=[x_{i,j},x_{i,j+1}]$ for $j\in \mathbb{Z}^0_k$. Let
\begin{equation}\label{SV bilinear form}
  a_{h}(v,w^*)=\sum_{i=1}^Na_{h,i}(v,w^*),\forall v\in \mathcal{H}_h,w^*\in\mathcal{V}_h,
\end{equation}
Then the SV scheme \eqref{SVschem} can be rewritten as: Find $u_h\in V_h$ such that
\begin{equation}\label{SVBFscheme}
  a_h(u_h,w_h^*)=(g,w_h^*),\forall w_h^*\in \mathcal{V}_h,
\end{equation}

\begin{remark} Note that the SV scheme \eqref{SVBFscheme} is consistent, that is, the exact solution $u$ of \eqref{eqn:para} also satisfies
\begin{equation}\label{consistentproperty}
  a_h(u,w_h^*)=(g,w_h^*), \forall w_h^*\in \mathcal{V}_h.
\end{equation}
In other words, the Galerkin orthogonality holds
\begin{equation}\label{Galerkin orthogonality}
  a_h(u-u_h,w_h^*)=0, \forall w_h^*\in \mathcal{V}_h.
\end{equation}
\end{remark}


\section{Stability analysis of RSV and LSV}

In this section, we  study the stability of the RSV and LSV methods.  To this end, we first relate the two SV schemes to some
Gauss-Legendre and  Radau numerical quadratures.
For any $f\in L^2([-1,1])$, we let
\begin{equation}\label{quadraturformula}
  Q_k[f]=\sum_{j=0}^{k+1}A_jf(s_j),~\text{and}~R[f]=\int_{-1}^1f(s)\text{d}s-Q_k[f],
\end{equation}
where $A_j=\int_{-1}^1l_j(s)\text{d}s, j\in\mathbb{Z}^0_{k+1}$ with $l_j$   the lagrange basis function at $s_j$.

\begin{itemize}
\item  For Gauss-Legendre quadrature, $\{s_j\}_{j=1}^k$ are $k$ Gauss points and $A_0=A_{k+1}=0$.
  Note that the Gauss-Legendre numerical quadrature is exact for polynomials of degree not more than $2k-1$, and the remainder of the quadrature is
     \begin{equation}\label{quadraturRemainder}
       R[f]=\frac{2^{2k+1}[k!]^4}{(2k+1)[(2k)!]^3}f^{(2k)}(\xi),\xi\in(-1,1).
     \end{equation}
\item For right Radau quadrature, $\{s_j\}_{j=1}^{k+1}$ are $k+1$  right Radau  points and $A_0=0$. The right Radau quadrature is exactly true for polynomials of degree not  more than $2k$ and the remainder of the quadrature is
     \begin{equation}\label{RRquadraturRemainder}
       R[f]=\frac{f^{(2k+1)}(\xi)}{(2k+1)!}\int_{-1}^1\prod_{j=1}^{j=k}(s-s_j)^2(s-1)\text{d}s,\xi\in(-1,1).
     \end{equation}

\item For left Radau quadrature,  $\{s_j\}_{j=0}^{k}$  are $k+1$    left Radau  points and $A_{k+1}=0$. The left Radau quadrature is exactly true for polynomials of degree not  more than $2k$.
     \begin{equation}\label{LRquadraturRemainder}
       R[f]=\frac{f^{(2k+1)}(\xi)}{(2k+1)!}\int_{-1}^1\prod_{j=1}^{j=k}(s-s_j)^2(s+1)\text{d}s,\xi\in(-1,1).
     \end{equation}
\end{itemize}

\subsection{Error equation}
 We  introduce a transformation $T$ from $V_h$ onto $\mathcal{V}_h$ as follows.
\begin{equation}\label{transformation}
  w^*:=Tw=\sum_{i=1}^{N}\sum_{j=0}^kw_{i,j}^*\chi_{\mathbf{V}_{i,j}}(x),w\in V_h,
\end{equation}
where  $\chi_A, A\subset [0,2\pi]$ is
the characteristic function defined as $\chi_A=1$ in $A$ and $\chi_A=0$ otherwise, and $w_{i,j}^*$ can be obtained by the following recurrence formula:
\begin{equation}\label{wij}
w_{i,0}^*=w_{i-\frac{1}{2}}^++A_{i,0}w_x(x_{i,0}), \ \ w_{i,j}^*=w_{i,j-1}^*+A_{i,j}w_x(x_{i,j}),j\in \mathbb{Z}_k,
\end{equation}
 with $A_{i,j}=\overline{h}_iA_j,  (i,j)\in \mathbb{Z}_N\times  \mathbb{Z}_k^0.$   By a direct calculation, we have
 \begin{equation*}
  w^*_{i,k}=\sum_{j=0}^{k}(w_{i,j}^*-w_{i,j-1})+w_{i,0}=w_{i-\frac{1}{2}}^++\sum_{j=0}^{k+1}A_{i,j}w_x(x_{i,j})- A_{i,k+1}w_x(x_{i+\frac{1}{2}}^-).
\end{equation*}
 Noticing that numerical quadrature is exact for $w_x$, we get
\begin{equation}\label{wik}
   w^*_{i,k} =w_{i+\frac{1}{2}}^--A_{i,k+1}w_x(x_{i+\frac{1}{2}}^-).
\end{equation}
\begin{remark} The values of $w_{i,0}^*, w_{i,k}^*$ are closely related to the choice of the points $s_j$. Specifically,
\begin{itemize}
\item If the partition points are chosen as Gauss-Legendre points, then $w_{i,0}^*=w_{i-\frac{1}{2}}^+$ and $w_{i,k}^*=w_{i+\frac{1}{2}}^-$.
\item If the partition points are chosen as right-Ladau points, then  $w_{i,0}^*=w_{i-\frac{1}{2}}^+$.
\item If the partition points are chosen as left-Ladau points, then  $w_{i,k}^*=w_{i+\frac{1}{2}}^-$.
\end{itemize}
\end{remark}

We have the following properties for the  transformation $T$.

\begin{lmm}
Suppose $T$ be the transformation defined by \eqref{transformation} and the partition points given by the quadrature abscissae described in \eqref{quadraturformula}, then the transformation $T$ is bijective and bounded, i.e.,
\begin{equation}\label{Tboundedproperty}
  \|Tw\|\lesssim \|w\|,\forall w\in V_h.
\end{equation}
 Here and in the following,  the notation $x\lesssim y$ means that $x$ can be bounded by $y$ multiplied by a
constants $C$, which is independent of the mesh size.
\end{lmm}
\begin{proof}
 We first show the transformation $T$ is injective. We assume that $Tw=0$, then we have $w_{i,j}=0, j\in\mathbb{Z}_k$,  and thus
 $w_x(x_{i,j})=0$ and either $w_{i-\frac{1}{2}}^{+}$ or $w_{i+\frac{1}{2}}^{-}=0$.
 As $w_x\in\mathcal{P}_{k-1}(\mathbf{V}_{i})$, we obtain that $w_x\equiv 0$. Therefore, we deduce that $w\equiv 0$. Suppose $\{\varphi_{i,j}\}_{j=0}^k$ is a basis of $V_h$, it is easy to prove that $\{T\varphi_{i,j}\}_{j=0}^k$ are linear independent, which yields that $\{T\varphi_{i,j}\}_{j=0}^k$ is a basis of $\mathcal{V}_h$.  Then $T$ is bijective.
By using the inverse inequality,  \eqref{Tboundedproperty} follows, which indicates that  $T$ is bounded.
\end{proof}

\begin{remark}
 We conclude from  Lemma 3.1 that the SV scheme \eqref{SVschem} is equivalent to: Find $u_h\in V_h$ such that
\begin{equation}\label{AlternativeSVBFscheme}
  a_h(u_h,w_h^*)=(g,w_h^*),\ \ \forall w_h\in V_h,\ \ w_h^*=Tw_h.
\end{equation}
\end{remark}


\begin{thm}\label{SVBAF}
Suppose $a_{h,i}(\cdot,\cdot)$ be the SV scheme defined by \eqref{SV bilinear form on element} with the partition points given by the quadrature abscissae described in \eqref{quadraturformula}. For any $v,w\in V_h$, let $V=\partial_x^{-1}v_t$ and $w^*=Tw$ be the peiecewise constant defined in \eqref{transformation}. Then
\begin{equation}\label{SVALterBF}
  a_{h,i}(v,w^*)=(v_t+(\alpha u)_x,w)_i+R_i[(V+\alpha v)w_x]+w_{i,0}^*(\alpha v^{+}-\alpha \widehat{v})|_{i-\frac{1}{2}}+w_{i,k}^*(\alpha\widehat{v}-\alpha v^-)|_{i+\frac{1}{2}}.
\end{equation}
\end{thm}
\begin{proof}
Since both $v\in V_h$ and $\alpha$  are continuous at the interior points,  we have from  \eqref{SV bilinear form on element} that
\begin{eqnarray}\nonumber
  a_{h,i}(v,w^*)
  &=&(v_t,w^*)_i+w_{i,0}^*(\alpha v|_{i,1}-\alpha \widehat{v}|_{i-\frac{1}{2}})+w_{i,k}^*(\alpha \widehat{v}|_{i+\frac{1}{2}}-\alpha v|_{i,k})+\sum_{j=0}^{k}w_{i,j}^*\int_{x_{i,j}}^{x_{i,j+1}}(\alpha v)_x\\ \label{SVDG}
  &=&\left(v_t+(\alpha v)_x,w^*\right)_i+w_{i,0}^*(\alpha v^{+}-\alpha \widehat{v})|_{i-\frac{1}{2}}+w_{i,k}^*(\alpha\widehat{v}-\alpha v^-)|_{i+\frac{1}{2}}.
\end{eqnarray}
 On the other hand,   a direct calculation from
 \eqref{wij} and \eqref{wik} yields
\begin{eqnarray}\nonumber
  (v,w^*)_i&=&\sum_{j=0}^k\int_{x_{x,j}}^{x_{i,j+1}} vw_{i,j}^*\text{d}x=\sum_{j=0}^kw_{i,j}^*(\partial_x^{-1}v|_{i,j+1}-\partial_x^{-1}v|_{i,j})\\ \nonumber
  &=& (w\partial_x^{-1}v)^-|_{i+\frac{1}{2}}-(w\partial_x^{-1}v)^+|_{i-\frac{1}{2}}-\sum_{j=1}^{k}\partial_x^{-1}v|_{i,j}(w_{i,j}^*-w_{i,j-1}^*)-\sum_{j=0,k+1}A_{i,j}w_x\partial_x^{-1}v|_{i,j}\\ \nonumber
  &=& w_{i+\frac{1}{2}}^-V^-|_{i+\frac{1}{2}}-w_{i-\frac{1}{2}}^+V^+|_{i-\frac{1}{2}}-\sum_{j=0}^{k+1}\partial_x^{-1}v|_{i,j}A_{i,j}w_x(x_{i,j})\\ \nonumber
  &=& \int_{x_{i-\frac{1}{2}}}^{x_{i+\frac{1}{2}}}(\partial_x^{-1}vw)_x\text{d}x-\int_{x_{i-\frac{1}{2}}}^{x_{i+\frac{1}{2}}}\partial_x^{-1}vw_x\text{d}x+R_i[\partial_x^{-1}vw_x],
\end{eqnarray}
where $R_i[f]$ denotes the residual between the exact integral of $f$ in  the element $\mathbf{V}_i$ and its numerical quadrature $Q_k[f]$.
Consequently,
\begin{equation}\label{vv}
  (v,w^*)_i=(v,w)_i+R_i[\partial_x^{-1}v\partial_x w].
\end{equation}
 Substituting \eqref{vv} into \eqref{SVDG} leads to the desired result \eqref{SVALterBF}.
\end{proof}

 We next provide a comparison
  between the SV scheme \eqref{SVBFscheme} and the discontinuous galerkin (DG) scheme, which is of great importance in our later stability analysis.
  The bilinear form of the
  DG  method  is defined by
  \begin{equation*}
  a^{DG}_h(v,w)=\sum_{i=1}^N a^{DG}_{h,i}(v,w),~\forall v,w\in\mathcal{H}_h,
\end{equation*}
and
\begin{equation}\label{DG BF}
a^{DG}_{h,i}(v,w)=(v_t,w)_i-(\alpha v,w_x)_i+\alpha \widehat{v}w^-|_{i+\frac{1}{2}}-\alpha \widehat{v}w^+|_{i-\frac{1}{2}}.
\end{equation}
Applying integration by part yields
\begin{equation}\label{difference}
  a^{DG}_{h,i}(v,w)=\left(v_t+(\alpha v)_x,w\right)_i+w_{i-\frac{1}{2}}^+(\alpha v^{+}-\alpha \widehat{v})|_{i-\frac{1}{2}}+w_{i+\frac{1}{2}}^-(\alpha\widehat{v}-\alpha v^-)|_{i+\frac{1}{2}}.
\end{equation}
Consequently,
\begin{equation}\label{RRSVbasi1}
  a_{h,i}(v,w^*)= a_{h,i}^{DG}(v,w)+R_i[(V+\alpha v)w_x]+\bar{R}_i[v,w],v\in V_h.
\end{equation}
where $V=\partial_x^{-1}v_t$ and
\begin{equation*}
  \bar{R}_i[v,w]=(w_{i,0}^*-w_{i-\frac{1}{2}}^+)(\alpha v^{+}-\alpha \widehat{v})|_{i-\frac{1}{2}}+(w_{i,k}^*-w_{i+\frac{1}{2}}^-)(\alpha\widehat{v}-\alpha v^-)|_{i+\frac{1}{2}}.
\end{equation*}

\subsection{Energy  stability}
To study the $L^2$ energy stability  for both LSV and RSV methods, we first estimate the  reminder term appeared in \eqref{RRSVbasi1}.
\begin{lmm}
For any $v\in V_h$, the reminder of  Radau-point or Gauss numerical quadrature satisfies
\begin{equation}\label{Ri}
\left|R_i[\alpha v v_x]\right|\lesssim \|v\|_{0,\mathbf{V}_i}^2,
\end{equation}
provided that $\alpha'\in C^0(\mathbf{V}_i).$  Furthermore, there holds for both LSV and  RSV methods
\begin{eqnarray}\label{Ri1}
 | \bar{R}_i[v,v]|  \lesssim  \|v\|_{0,\mathbf{V}_{i}}(\|v\|_{0,\mathbf{V}_{i}}+\|v\|_{0,\mathbf{V}_{i+1}}+\|v\|_{0,\mathbf{V}_{i-1}}).
\end{eqnarray}

\end{lmm}
\begin{proof}Let   $\bar{\alpha}_i$ be the cell average of $\alpha$ on $\mathbf{V}_i$.   Then

\begin{eqnarray}\nonumber
  R_i[\alpha v v_x]
   &=& \int_{x_{i-\frac{1}{2}}}^{x_{i+\frac{1}{2}}}\alpha v v_x\text{d}x-\sum_{j=0}^{k+1}A_{i,j}\left[\alpha(x_{i,j})-\bar{\alpha}_i\right]v(x_{i,j}) v_x(x_{i,j})-\sum_{j=0}^{k+1}A_{i,j}\bar{\alpha}_iv(x_{i,j}) v_x(x_{i,j})\\ \label{remainder}
   &=&\int_{x_{i-\frac{1}{2}}}^{x_{i+\frac{1}{2}}}(\alpha- \bar{\alpha}_i)v v_x\text{d}x-\sum_{j=0}^{k+1}A_{i,j}\left[\alpha(x_{i,j})-\bar{\alpha}_i\right]v(x_{i,j}) v_x(x_{i,j}): =R_{i,1}-R_{i,2}.
\end{eqnarray}
Here in the last step, we use that the Radau or Gauss numerical quadrature is exact for $2k-1$ degree polynomials. By the Cauchy-Schwarz inequality and inverse inequality, we obtain that
\begin{equation}\label{remainder R1}
  |R_{i,1}|\lesssim h_i\|\alpha\|_{0,\infty,\mathbf{V}_i}\|v\|_{0,\mathbf{V}_i}\|v_x\|_{0,\mathbf{V}_i}\lesssim\|v\|_{0,\mathbf{V}_i}^2.
\end{equation}
 Since the right or left Radau numerical quadrature is exact for polynomial of degree of $2k$, we have for Radau numerical quadrature that
\[
 \left(\sum_{j=0}^{k+1}A_{i,j}v^2(x_{i,j})\right)^{1/2}= \|v\|_{0,\mathbf{V}_i}.
\]
 While for Gauss  numerical quadrature, there holds
\begin{equation*}
 \left(\sum_{j=0}^{k+1}A_{i,j}v^2(x_{i,j})\right)^{1/2}=\left(\|v\|_{0,\mathbf{V}_i}^2-R_i[v^2]\right)^{1/2}
\end{equation*}
with $R_i[v^2]$ the Gauss numerical quadrature error, i.e.,
\begin{equation*}
  R_i[v^2]=\bar{h}_i^{2k} \frac{2^{2k+1}(k!)^2}{2(2k+1)[(2k)!]^3}\|v^{(k)}\|_{0,\mathbf{V}_i}\lesssim \|v\|_{0,\mathbf{V}_i}^2,
\end{equation*}
 Consequently, for both Radau or Gauss numerical quadrature, there holds
\begin{equation}\label{remainder R2}
  |R_{i,2}|\lesssim h_i\|\alpha_x\|_{0,\infty,\mathbf{V}_i}\left(\sum_{j=0}^{k+1}A_{i,j}v^2(x_{i,j})\right)^{1/2}\left(\sum_{j=0}^{k+1}A_{i,j}v_x^2(x_{i,j})\right)^{1/2}\lesssim\|v\|_{0,\mathbf{V}_i}^2.
\end{equation}
Substituting  \eqref{remainder R1} and \eqref{remainder R2} into  \eqref{remainder} yields the desired result
 \eqref{Ri}.

 For the   term $\bar{R}_i[v,v]$,
if $\alpha(x_{i-\frac{1}{2}})\alpha(x_{i+\frac{1}{2}})\ge 0$,  then a direct calculation yields that $\bar{R}_i[v,v]=0$;
if $\alpha(x_{i-\frac{1}{2}})\alpha(x_{i+\frac{1}{2}})<0$,  then there exists at least one point $\eta_i\in \mathbf{V}_i$ satisfying $\alpha(\eta_i)=0$ and thus
$\|\alpha\|_{0,\infty,V_i}\lesssim h$.    Consequently,  we have from  \eqref{wij} and \eqref{wik} that
\begin{eqnarray}\nonumber
 | \bar{R}_i[v,v]|&\lesssim & h \|v_x\|_{0,\infty,\mathbf{V}_i}\|\alpha\|_{0,\infty, \mathbf{V}_i}(  |[v]_{i-\frac{1}{2}}|+|[v]_{i+\frac{1}{2}}|) \\ \nonumber
  &\lesssim & \|v\|_{0,\mathbf{V}_{i}}(\|v\|_{0,\mathbf{V}_{i}}+\|v\|_{0,\mathbf{V}_{i+1}}+\|v\|_{0,\mathbf{V}_{i-1}}),
\end{eqnarray}
  where in the last step, we have used the inverse inequality.  Then \eqref{Ri1} follows.
 \end{proof}

  To study the stability of  the LSV method,
we   still need to  introduce a new norm $|\|\cdot|\|$ which is equivalent with the $L^2$-norm.
The norm $|\|\cdot|\|$ is defined by
\begin{equation}\label{newnorm}
  |\|v|\|^2=\sum_{i=1}^N|\|v|\|_i^2,~|\|v|\|_i^2=(v,v^*)_i, \forall v\in V_h.
\end{equation}
where $v^*=Tv$.
In light of \eqref{vv}, we easily obtain
 \begin{equation}\label{vv1}
  |\|v|\|_i^2=(v,v^*)_i=(v,v)_i+R_i[\partial_x^{-1}v\partial_x v].
\end{equation}
 Using the error of Gauss numerical quadrature, we have
\begin{equation*}
  R_i[\partial_x^{-1}v\partial_x v]=\bar{h}_i^{2k+1} \frac{2^{2k+1}(k!)^4}{(2k+1)[(2k)!]^3}\frac{d^{2k}}{x^{2k}}(\partial_x^{-1}vv_x)=\bar{h}_i^{2k} \frac{2^{2k+1}k(k!)^2}{(2k+2)(2k+1)[(2k)!]^3}(v^{(k)},v^{(k)})_i.
\end{equation*}
By using the inverse inequality, we have $0\le R_i[\partial_x^{-1}v\partial_x v] \lesssim \|v\|_{0,\mathbf{V}_i}$. Thus the norm $|\|\cdot|\|$ defined in \eqref{newnorm} is equivalent with the $L^2$-norm.

Now we are ready to show the stability result of the two SV methods.

\begin{thm}\label{Energystablity}
Let $a_{h}(\cdot,\cdot)$ be the SV bilinear form defined in \eqref{SV bilinear form}, and $u_h$ be the solution of  \eqref{SVBFscheme}.  Then   for  both RSV and LSV methods,
\begin{equation}\label{stablityRRSV}
 \frac{1}{2}\frac{d}{dt}\|u_h\|^2_E\lesssim -\frac{1}{2}(\alpha_xu_h,u_h)+a(u_h,u_h^*)+\|u_h\|_{0}^2.
\end{equation}
Here $\|u_h\|_E=\|u_h\|_0$ for RSV method and $\|u_h\|_E=\||u_h\||$ for LSV methods.
 Consequently, both RSV and LSV methods are stable in the sense that
\begin{equation}\label{stablityRRSV1}
  \|u_h(\cdot,t)\|_E\lesssim   \|u_h(\cdot,0)\|_E,\ \ t\in (0,T].
\end{equation}

\end{thm}

\begin{proof} First, we use the energy inequality of the DG method in
\cite{Cao-Shu-ZhangM2AN}   to get
\begin{equation}\label{stablityDG}
  (v_t,v)_i\leq a_{h,i}^{DG}(v,v)-\frac{1}{2}(\alpha_x v,v)_i, \forall v\in V_h.
\end{equation}
Second, by choosing $w=v$ in \eqref{RRSVbasi1}, we have
\begin{equation}\label{RRSVbasi}
  a_{h,i}(v,v^*)= a_{h,i}^{DG}(v,v)+R_i[\partial_x^{-1}vv_t]+R_i[\alpha vv_x]+\bar{R}_i[v,v],v\in V_h.
\end{equation}
Then we conclude from \eqref{Ri} and \eqref{Ri1} that
Therefore,
\begin{equation}\label{RRSVbasionElement}
  a^{DG}_{h,i}(v,v)+ R_i[\partial_x^{-1}vv_t]\lesssim a_{h,i}(v,v^*)+\|v\|_{0,\mathbf{V}_{i}}(\|v\|_{0,\mathbf{V}_{i}}+\|v\|_{0,\mathbf{V}_{i+1}}+\|v\|_{0,\mathbf{V}_{i-1}}), \forall v\in V_h,
\end{equation}

As for the  term $R_i[\partial_x^{-1}vv_t]$ in \eqref{RRSVbasi},  noticing that
$v_x\partial_x^{-1}v_t\in \mathcal{P}_{2k}$, we use the property of Radau-point numerical quadrature to get $R_i[v_x\partial_x^{-1}v_t]=0$ for RSV.
While for LSV, we have from \eqref{vv1} that
\[
  R_i[v_x\partial_x^{-1}v_t]=(v_t,v^*)_i-(v_t,v)_i.
\]
 which yields, together with \eqref{stablityDG} and \eqref{RRSVbasionElement} that
\begin{eqnarray}\label{RRSVbasionAll}
  && (v_t,v)\lesssim a_{h}(v,v^*)-\frac{1}{2}(\alpha_x v,v)+\|v\|_0^2, \ {\rm for\ RSV},\\\label{LSVbasionAll}
  &&  (v_t,v^*)\lesssim a_{h}(v,v^*)-\frac{1}{2}(\alpha_x v,v)+\|v\|_0^2, \ {\rm for\ LSV}.
\end{eqnarray}
Then \eqref{stablityRRSV} follows by choosing $v=u_h$.
By   applying Gronwall inequality and the fact that $a_{h}(u_h,v^*)=0$ for all $v\in V_h$,
we obtain the  stability result \eqref{stablityRRSV1}  directly.
\end{proof}

\subsection{Flux function stability}

Following the strategy presented in \cite{Cao-Shu-ZhangM2AN}, we would like to study another stability inequality,  i.e., the stability for the flux function $\alpha u_h$.
 Towards this end,  we first define  $\tilde{\alpha}_i^2:=\alpha^2(x_{i-\frac{1}{2}})+\alpha^2(x_{i+\frac{1}{2}})$ and let
\begin{equation*}
A_{h,i}(v,w^*)=\tilde{\alpha}_i^2a_{h,i}(v,w^*),~A_h(v,w^*)=\sum_{i=1}^NA_{h,i}(v,w^*),v,w\in H_h^1.
\end{equation*}
Correspondingly, we can define $\|\cdot\|_\alpha$ and $|\|\cdot|\|_\alpha$ norms as follows:
\begin{equation*}
  \|v\|_\alpha^2: =\sum_{i=1}^N\tilde{\alpha}_i^2\|v\|_{0,\mathbf{V}_i}^2~\text{and}~|\|v|\|_\alpha^2: =\sum_{i=1}^N\tilde{\alpha}_i^2|\|v|\|_{\mathbf{V}_i}^2.
\end{equation*}
  It has been proved in \cite{Cao-Shu-ZhangM2AN} that
\begin{equation}\label{new stablity for DG}
  \frac{1}{2}\frac{d}{dt}\|v\|_{\alpha}^2\lesssim |A_h^{DG}(v,v)|+\|v\|_{\alpha}^2, v\in V_h,
\end{equation}
where
\begin{equation*}
A_{h,i}^{DG}(v,w)=\tilde{\alpha}_i^2a_{h,i}^{DG}(v,w),~A_h^{DG}(v,w)=\sum_{i=1}^NA^{DG}_{h,i}(v,w),v,w\in H_h^1.
\end{equation*}

 Following the same argument as what we did in Theorem \ref{Energystablity},
 we have the following stability result for the flux function.

\begin{thm}
 Both the RSV and LSV methods are stable in the flux function norm. That is, there holds for RSV
\begin{equation}\label{new stablity for RSV}
 \frac{1}{2} \frac{d}{dt}\|v\|_{\alpha}^2\lesssim |A_h(v,v^*)|+\|v\|_{\alpha}^2, \ \forall v\in V_h,
\end{equation}
and for the LSV scheme
\begin{equation}\label{new stablity for LSV}
 \frac{1}{2}\frac{d}{dt}|\|v|\|_{\alpha}^2\lesssim |A_h(v,v^*)|+\||v\||_{\alpha}^2,\ \ \forall v\in V_h.
\end{equation}
\end{thm}
Here we omit the proof since it is similar to that of Theorem \ref{Energystablity}.

\section{Optimal  error estimates for RSV and LSV methods} In this section, we focus our attention on the analysis of
 the optimal  error estimates for both RSV and LSV methods.  We begin with the introduction of some interpolation functions.

Given a function $\phi\in \mathcal{H}_h$, we denote by $I_h^{+}\phi, I_h^{-}\phi\in V_h$ the standard Lagrange interpolation of $\phi$, which satisfy the following $k+1$ conditions
on each volume-element $\mathbf{V}_i(i\in Z_N)$:
\begin{equation}\label{ interpolationL}
  ~~~~~I_h^{-}\phi(x_{i,j})=\phi(x_{i,j}),~j=1,2,\cdots,k+1,
\end{equation}
\begin{equation}\label{ interpolationR}
  I_h^{+}\phi(x_{i,j})=\phi(x_{i,j}),~j=0,1,\cdots,k.
\end{equation}
In addition, we also define the Lagrange interpolation $I^{\pm}_h$ as follows:
\begin{equation}\label{ interpolationLR}
  I_h^{\pm}\phi(x_{i-\frac{1}{2}})=\phi(x_{i-\frac{1}{2}}^{+}),I_h^{\pm}\phi(x_{i+\frac{1}{2}})=\phi(x_{i+\frac{1}{2}}^{-}),I_h^{\pm}\phi(x_{i,j})=\phi(x_{i,j}),~j\in \mathbb{Z}_{k-1}.
\end{equation}
The standard interpolation theory shows that
\begin{equation}\label{interpolation theory}
  \|I_h^{-}\phi-\phi\|_0+\|I_h^{+}\phi-\phi\|_0+\|I_h^{\pm}\phi-\phi\|_0\lesssim h^{k+1}|\phi|_{k+1}.
\end{equation}
 With the above three interpolation operators, we  define a particular interpolation $I_h\phi$ of $\phi$ as follows:
\begin{eqnarray}\label{Newinterpolation}
I_h\phi=\left \{
\begin{array}{ll}
I_h^{-}\phi,   &if~\alpha|_{i-\frac{1}{2}}\geq 0,\alpha|_{i+\frac{1}{2}}> 0,\\
I_h^{+}\phi,   &if~\alpha|_{i-\frac{1}{2}}\leq 0,\alpha|_{i+\frac{1}{2}}< 0, \\
I_h^{\pm}\phi, &otherwise,\\
\end{array}
\right.
\end{eqnarray}
 We have the  following property of $I_h\phi$.
\begin{lmm} For any  $\phi\in C^0[0,2\pi]$,  the  interpolation function $I_h\phi$
 is exact in numerical flux, i.e,
\begin{equation}\label{fluxproperty}
  \widehat{I_h\phi}|_{i+\frac{1}{2}}=\phi(x_{i+\frac{1}{2}}), \forall i\in \mathbb{Z}_N.
\end{equation}
\end{lmm}
\begin{proof}
For any $i\in \mathbb{Z}_N$, by the definition of numerical flux and interpolation $I_h$, it holds that for $\alpha(x_{i+\frac{1}{2}})>0$,
\begin{equation*}
  \widehat{I_h\phi}|_{i+\frac{1}{2}}=I_h\phi(x_{i+\frac{1}{2}}^-)=I_h^-\phi(x_{i+\frac{1}{2}}^-)~\text{or}~I_h^{\pm}\phi(x_{i+\frac{1}{2}}^-)=\phi(x_{i+\frac{1}{2}}),
\end{equation*}
and for the other case $\alpha(x_{i+\frac{1}{2}})\leq 0$,
\begin{equation*}
  \widehat{I_h\phi}|_{i+\frac{1}{2}}=I_h\phi(x_{i+\frac{1}{2}}^+)=I_h^+\phi(x_{i+\frac{1}{2}}^+)~\text{or}~I_h^{\pm}\phi(x_{i+\frac{1}{2}}^+)=\phi(x_{i+\frac{1}{2}}),
\end{equation*}
which both implies that the numerical flux is exact.
\end{proof}

Recalling the  definition of $a_{h,i}(\cdot,\cdot)$  in \eqref{SV bilinear form on element}, we  divide $a_{h,i}(\cdot,\cdot)$ into two parts. That is, for all $ v\in \mathcal{H}_h,w^*\in\mathcal{V}_h,i\in \mathbb{Z}_N$,  denote
$a_{h,i}(v,w^*)=b_{i,1}(v,w^*)+b_{i,2}(v,w^*)$  with
\begin{eqnarray*}
 b_{i,1}(v,w^*)=\sum_{j=0}^kw_{i,j}^*\int_{x_{i,j}}^{x_{i,j+1}}v_t\text{d}x,\  \
b_{i,2}(v,w^*)=\sum_{j=0}^kw_{i,j}^*\left(\alpha \widehat{v}|_{i,j+1}-\alpha \widehat{v}|_{i,j}\right).
\end{eqnarray*}

\begin{lmm}
    Let $u(\cdot,t)\in H^{k+2}(\Omega)$ for any $t\in[0,T]$ be the solution of \eqref{eqn:para},   $I_hu$ be the interpolation function of $u$ defined in \eqref{Newinterpolation}. Then
\begin{equation}\label{Newinterpolationproperty}
  |b_{i,2}(u-I_hu,v^*)|\lesssim h^{k+\frac{3}{2}}\|u\|_{k+1,\infty}\|v\|_{0,\mathbf{V}_i}, \forall v\in V_h, i\in \mathbb{Z}_N.
\end{equation}
\end{lmm}

\begin{proof}
Rearranging the items of $b_{i,2}(v,w^*)$, we get
\begin{equation}\label{newb2form}
  b_{i,2}(v,w^*)=w_{i+\frac{1}{2}}^{-}\alpha \widehat{v}|_{i+\frac{1}{2}}-w_{i-\frac{1}{2}}^{+}\alpha\widehat{v}|_{i-\frac{1}{2}}-\sum_{j=1}^k\alpha(x_{i,j}) v(x_{i,j})A_{i,j}w_x(x_{i,j}).
\end{equation}
Choosing $v=u-I_hu$ in \eqref{newb2form} and using  \eqref{fluxproperty} yields
\begin{equation}
  b_{i,2}(u-I_hu,v^*)=-\sum_{j=1}^k\alpha(x_{i,j}) (u-I_hu)(x_{i,j})A_{i,j}v_x(x_{i,j}).
\end{equation}
If $I_h=I_h^-$ or $I_h^{+}$, we get immediately that $(u-I_hu)(x_{i,j})=0$ for all $ j\in \mathbb{Z}_k$, which implies $b_{i,2}(u-I_hu,v^*)=0$. If $I_h=I_h^{\pm}$,
 then  $\alpha_{i+\frac 12}\alpha_{i-\frac 12}<0$, which implies that
 there exists at least  a $\eta_i\in \mathbf{V}_i$ satisfying $\alpha(\eta_i)=0$.  Therefore,
\begin{equation*}
  b_{i,2}(u-I_hu,v^*)=(\alpha(x_{i,k})-\alpha(\eta_i)) (u-I_hu)(x_{i,k})A_{i,k}v_x(x_{i,k}).
\end{equation*}
   By using the approximation property of the interpolation function and the inverse inequality, we get
 \begin{equation}\label{Interproperty1}
  \left | b_{i,2}(u-I_hu,v^*)\right |\lesssim h^{k+3} \cdot \|u\|_{k+1,\infty} \cdot\|v\|_{1,\infty,\mathbf{V}_i}\lesssim h_i^{k+\frac{3}{2}}\|u\|_{k+1,\infty}\|v\|_{0,\mathbf{V}_i}.
\end{equation}
   This finishes our proof.
\end{proof}

Now, we are ready to give the optimal error estimate for $\|u_h-I_hu\|_0$.
\begin{thm}\label{ConvegenceForRRSV}
Let $u(\cdot,t)\in H^{k+2}$ for any $t\in[0,T]$ be the solution of \eqref{eqn:para}, $I_hu$ be the interpolation function of $u$   defined in \eqref{Newinterpolation}, and $u_h$ be the solution of the scheme \eqref{AlternativeSVBFscheme} with the initial solution $u_h(x,0)=I_hu_0$. Then for both the RSV and LSV methods,
\begin{equation}\label{L2ConvegenceForRRSV}
  \|(u-u_h)(\cdot,t)\|_0\lesssim h^{k+1} \sup_{\tau\in[0,t]}\|u(\cdot,\tau)\|_{k+2}.
\end{equation}
\end{thm}
\begin{proof}
In each element $\mathbf{V}_i$, we have
\begin{equation}\label{134}
  a_{h,i}(u-I_hu,v_h^*)=b_{i,1}(u-I_hu,v_h^*)+b_{i,2}(u-I_hu,v_h^*)=(u_t-(I_hu)_t,v_h^*)+b_{i,2}(u-I_hu,v_h^*).
\end{equation}
By using \eqref{Tboundedproperty}, \eqref{Newinterpolationproperty},  and  the fact that $(I_hu)_t=I_hu_t$, we obtain that for any $v_h\in V_h$ and $i\in \mathbb{Z}_N$,
\begin{eqnarray}\nonumber
  |a_{h,i}(u-I_hu,v_h^*)|&\lesssim&\|u_t-I_hu_t\|_{0,\mathbf{V}_i}\|v_h^*\|_{0,\mathbf{V}_i}+h^{k+\frac{3}{2}}\|u\|_{k+1,\infty}\|v\|_{0,\mathbf{V}_i}\\ \label{124}
  &\lesssim& h_i^{k+1}(\|u_t\|_{k+1,\mathbf{V}_i}+h_i^{\frac{1}{2}}\|u\|_{k+1,\infty})\|v\|_{0,\mathbf{V}_i}.
\end{eqnarray}
Choosing $v=u_h-I_hu$ in \eqref{RRSVbasionAll} and using the Galerkin orthogonality \eqref{Galerkin orthogonality}, we have for RSV methods
\begin{eqnarray}\nonumber
  ((u_h-I_hu)_t, u_h-I_hu)
  &\lesssim& |a_{h}(u-I_hu, u_h-I_hu)|+\|u_h-I_hu\|_0^2\\
  &\lesssim& h^{k+1}\|u\|_{k+2}+\|u_h-I_hu\|_0^2.
\end{eqnarray}
  Here in the last step, we have used the fact $u_t=-u_x$ and  \eqref{124}.  Similarly,  for LSV methods,
   we take  $v=u_h-I_hu$ in \eqref{LSVbasionAll}  and use \eqref{124} to derive
 \begin{eqnarray}\nonumber
  ((u_h-I_hu)_t,(u_h-I_hu)^*)
  \lesssim  h^{k+1}\|u\|_{k+2}+\|u_h-I_hu\|_0^2.
\end{eqnarray}
 By applying the Gronwall inequality and initial value, we have for both the RSV and LSV that
 \[
     \|(u_h-I_hu)(\cdot,t)\|_0\lesssim h^{k+1}\|u\|_{k+2}
 \]
 Then the desired results follows from the
 the standard interpolation approximation \eqref{interpolation theory}.
\end{proof}

\section{Superconvergence  of RSV and LSV methods}
 This section is dedicated to the superconvergence analysis  for RSV and LSV methods.
 Both the superconvergence  properties of the flux function approximation and the solution itself approximation  will be investigated.

We begin with some preliminaries. First,
we divide the whole domain $\Omega$ into three parts $\Omega=\Omega_1\cup\Omega_2\cup\Omega_3$, where
\begin{equation*}
  \Omega_1=\{\mathbf{V}_i, \alpha(x_{i-\frac{1}{2}})>0,\alpha(x_{i+\frac{1}{2}})>0\},\Omega_2=\{\mathbf{V}_i, \alpha(x_{i-\frac{1}{2}})<0,\alpha(x_{i+\frac{1}{2}})<0\},\Omega_3=\Omega\setminus(\Omega_1\cup\Omega_2).
\end{equation*}
Second,  we define for all  $v,w\in H_h^1$,
\begin{equation}\label{BForm}
 b(v,w^*)=\sum_{i=1}^N b_i(v,w^*),~b_i(v,w^*)=(v_t,w^*)_i-\sum_{j=1}^kA_{i,j}\alpha(x_{i,j})w_x(x_{i,j})v(x_{i,j}).
\end{equation}
 In light of \eqref{SV bilinear form on element} and \eqref{fluxproperty}, we have that
\[
    a_{h,i}(u-I_hu,v_h^*)=b_i(u-I_hu,v_h^*), \ \ A_{h,i}(u-u_I,v_h^*)=\tilde{\alpha}_i^2b_i(u-I_hu,v_h^*),
\]
  and thus,
\begin{equation}\label{AForm}
  A_h(u-I_hu, v_h^*)=\sum_{{\bf V}_i\in\Omega_1}\tilde{\alpha}_i^2 b_{i}(u-I_h^-u,v_h^*)+\sum_{{\bf V}_i\in\Omega_2}\tilde{\alpha}_i^2b_{i}(u-I_h^+u,v_h^*)+\sum_{{\bf V}_i\in\Omega_3}\tilde{\alpha}_i^2b_{i}(u-I_h^{\pm}u,v_h^*).
\end{equation}
Noticing that for any $\mathbf{V}_i\in\Omega_1$, $(u-I_h^-u)(x_{i,j})=0(j\in \mathbb{Z}_k)$ and for $\mathbf{V}_i\in\Omega_2$, $(u-I_h^+u)(x_{i,j})=0(j\in \mathbb{Z}_k)$, we derive that
\begin{equation}\label{AForm1}
  A_h(u-I_hu,v_h^*)=\sum_{{\bf V}_i\in\Omega_1}\tilde{\alpha}_i^2((u-I_h^-u)_t,v_h^*)+\sum_{{\bf V}_i\in\Omega_2}\tilde{\alpha}_i^2((u-I_h^+u)_t,v_h^*)+\sum_{{\bf V}_i\in\Omega_3}\tilde{\alpha}_i^2b_{i}(u-I_h^{\pm}u,v_h^*).
\end{equation}

Following \cite{caozhangzou2014}, we begin our superconvergence analysis by constructing a local correction function for $((u-I_h^-u)_t,v_h)$ on $\Omega_1$ and for $((u-I_h^+u)_t,v_h)$ on $\Omega_2$, respectively. Let
\begin{equation}\label{Aadefination}
  w_0=u-I_h^-u,\ \  \bar w_0=u-I_h^+u,\ \ \overline{\alpha}_i=\max_{x\in {\bf V}_i}|\alpha(x)|,\ \ \mathcal{P}_{-}({\bf V}_i):=\mathcal{P}_k({\bf V}_i)\setminus \mathcal{P}_0({\bf V}_i).
\end{equation}
We define the correction function $w_1\in V_h$ for $((u-I_h^-u)_t,v^*_h)$  as follows:
\begin{equation}\label{local correction w1}
 w_1(x_{i+\frac{1}{2}}^{-})=0,~(\overline{\alpha}_i w_1,v_x)_i=(w_0,v^*),~\forall v\in \mathcal{P}_{-}({\bf V}_i).
\end{equation}
  We have the following properties for $w_1$.
\begin{lmm}\label{lemma:1}
Let $\mathbf{V}_i\in\Omega_1$, $w_1\in V_h$ be defined by \eqref{local correction w1}  with $\overline{\alpha}_i$ given by \eqref{Aadefination}.   Then
 for both RSV and LSV  methods,
\begin{equation}\label{Correctionproperty}
  b_{i}(w_0+w_1,v^*)=(\partial_tw_1,v^*)_i-\sum_{j=1}^kA_{i,j}(\alpha(x_{i,j})-\overline{\alpha}_i)v_x(x_{i,j})w_1(x_{i,j}), \forall v\in \mathcal{P}_{-}({\bf V}_i).
\end{equation}
Furthermore, if $u\in W^{k+3,\infty}$, then
\begin{equation}\label{correctionestimatew1}
  \|w_1\|_{0,{\bf V}_i}+\|\partial_tw_1\|_{0,{\bf V}_i}\lesssim \frac{h_i^{k+\frac{5}{2}}}{\|\alpha\|_{0,\infty,{\bf V}_i}} \| u\|_{k+3,\infty,{\bf V}_i}.
\end{equation}
Consequently,
\begin{equation}\label{superconvergencecorrection1}
  \tilde{\alpha}_i^2b_{i}(w_0+w_1,v^*)\lesssim h_i^{k+\frac{5}{2}} \| u\|_{k+3,\infty,{\bf V}_i} \|v\|_{\alpha,\mathbf{V}_i},\forall v\in \mathcal{P}_{-}({\bf V}_i).
\end{equation}
\end{lmm}
\begin{proof}  First,  for any
$\mathbf{V}_i\in\Omega_1$, the partial points $x_{i,j}$ are chosen as Gauss or right Radau points.  By the property of Gauss or right Radau numerical quadrature, we have
\[
   \sum_{j=1}^kA_{i,j}v_x(x_{i,j})w_1(x_{i,j})=\sum_{j=1}^{k+1}A_{i,j}v_x(x_{i,j})w_1(x_{i,j})=(v_x,w_1)_i,\ \ v\in V_h.
\]
Then a direct calculation
from \eqref{BForm} and \eqref{local correction w1}  yields that for any $v\in \mathcal{P}_{-}({\bf V}_i)$,
\begin{eqnarray}\nonumber
   b_{i}(w_0+w_1,v^*)
   &=& (\partial_t w_0,v^*)_i+(\partial_tw_1,v^*)_i-\sum_{j=1}^kA_{i,j}\alpha(x_{i,j})v_x(x_{i,j})w_1(x_{i,j}) \\ \nonumber
   &=& (\partial_t w_0,v^*)_i+(\partial_tw_1,v^*)_i-\sum_{j=1}^kA_{i,j}(\alpha(x_{i,j})-\overline{\alpha}_i)v_x(x_{i,j})w_1(x_{i,j})-(\overline{\alpha}_iw_1,v_x) \\ \nonumber
   &=& (\partial_tw_1,v^*)_i-\sum_{j=1}^kA_{i,j}(\alpha(x_{i,j})-\overline{\alpha}_i)v_x(x_{i,j})w_1(x_{i,j}).
\end{eqnarray}
 This finishes the proof of \eqref{Correctionproperty}.

To estimate $w_1$, we suppose $w_1$ has the  following Legendre expansion in each element $\mathbf{V}_i$
\begin{equation*}
  w_1|_{\mathbf{V}_i}=\sum_{j=0}^kc_{i,j}(t)L_{i,j}(x).
\end{equation*}
 Here   $L_{i,j}$ denotes the Legendre polynomial of degree $j$ in ${\bf V}_i$.
 Denoting  $\phi_{i,j+1}=\int_{x_{i-\frac{1}{2}}}^{x}L_{i,j}(s)\text{d}s$ and  choosing $v=\phi_{i,m+1}(m\in \mathbb{Z}_{k-1}^0)$ in \eqref{local correction w1} leads to
\begin{equation*}
  |c_{i,m}|=\frac{(2m+1)}{\overline{\alpha}_ih_i}|(\partial_tw_0,\phi^*_{i,m+1})_i|\lesssim \frac{1}{\overline{\alpha}_i} \|\partial_tw_0\|_{0,\mathbf{V}_i}\|\phi_{i,m+1}\|_{0,\mathbf{V}_i}\lesssim \frac{h_i}{\overline{\alpha}_i}\|\partial_tw_0\|_{0,\infty,\mathbf{V}_i}.
\end{equation*}
 Using $w_1(x_{i+\frac{1}{2}}^{-})=0$, we obtain that
\begin{equation}\label{abslutewproperty}
  |c_{i,k}|=|\sum_{j=0}^{k-1}c_{i,j}|\lesssim\frac{h_i}{\overline{\alpha}_i}\|\partial_tw_0\|_{0,\infty,\mathbf{V}_i}.
\end{equation}
Similarly,  there holds for all $m\in \mathbb{Z}_{k}^0$,
\begin{equation*}
  |\partial_t c_{i,m}| \lesssim \frac{h_i}{\overline{\alpha}_i}\|\partial_{tt}w_0\|_{0,\infty,\mathbf{V}_i}.
\end{equation*}
Consequently,
\begin{equation*}
  \|w_1\|_{0,\mathbf{V}_i}^2\lesssim h_i\sum_{j=0}^kc_{i,j}^2\lesssim \frac{h_i^{3}}{\overline{\alpha}_i^2}\|\partial_tw_0\|_{0,\infty,\mathbf{V}_i}^2,\ \
  \|\partial_tw_1\|_{0,\mathbf{V}_i}^2\lesssim h_i\sum_{j=0}^k(\partial_tc_{i,j})^2\lesssim
  \frac{h_i^{3}}{\overline{\alpha}_i^2}\|\partial_{tt}w_0\|_{0,\infty,\mathbf{V}_i}^2.
\end{equation*}
 Then    \eqref{correctionestimatew1} follows from the approximation property of the interpolation function.

As a direct consequence of Cauchy-Schwarz inequality and  \eqref{correctionestimatew1},
\begin{equation}\label{estimaiteformbi1}
  \tilde{\alpha}_i^2(\partial_tw_1,v^*)_i\lesssim \|\partial_tw_1\|_{\alpha}\|v\|_{\alpha}\lesssim h_i^{k+\frac{5}{2}}\| u\|_{k+3,\infty,{\bf V}_i}\|v\|_{\alpha,\mathbf{V}_i}.
\end{equation}
 Similarly, we use the inverse inequality to derive
\begin{eqnarray}\nonumber
  \tilde{\alpha}_i^2|A_{i,j}(\alpha(x_{i,j})-\overline{\alpha}_i)v_x(x_{i,j})w_1(x_{i,j})| &\lesssim& \tilde{\alpha}_i^2 h_i^2\|v_x\|_{0,\infty,\mathbf{V}_i}\|w_1\|_{0,\infty,\mathbf{V}_i}\lesssim \tilde{\alpha}_i^2 \|v\|_{0,\mathbf{V}_i}\|w_1\|_{0,\mathbf{V}_i} \\\label{estimaiteformbi2}
   &\lesssim& h_i^{k+\frac{5}{2}} \| u\|_{k+2,\infty,{\bf V}_i}\|v\|_{\alpha,\mathbf{V}_i}.
\end{eqnarray}
Combining  \eqref{Correctionproperty}, \eqref{estimaiteformbi1} and \eqref{estimaiteformbi2} together, we get \eqref{superconvergencecorrection1} immediately.
\end{proof}

Next, we construct the correction function for the term $((u-I_h^+u)_t,v_h^*)$. In each element $\mathbf{V}_i\in\Omega_2$,  let $w_2\in V_h$ satisfy
\begin{equation}\label{local correction w2}
 w_2(x_{i-\frac{1}{2}}^{+})=0,~(\overline{\alpha}_i w_2,v_x)_i=(\overline{w}_0,v^*),~\forall v\in \mathcal{P}_{-}({\bf V}_i).
\end{equation}
Here $\alpha_i$ and $\mathcal{P}_{-}({\bf V}_i)$ are the same as in \eqref{Aadefination}.  Following  the same  argument as that in
 Lemma \ref{lemma:1}, we  have the following results for $w_2$.

\begin{lmm}Let $\mathbf{V}_i\in\Omega_2$, $w_2\in V_h$ be defined by \eqref{local correction w2} with $\overline{\alpha}_i$ given in \eqref{Aadefination}.   Then
for both   RSV and LSV  methods,
\begin{equation}\label{Correctionpropertyw21}
  b_{i}(\overline{w}_0+w_2,v^*)=(\partial_tw_2,v^*)_i-\sum_{j=1}^kA_{i,j}(\alpha(x_{i,j})-\overline{\alpha}_i)v_x(x_{i,j})w_2(x_{i,j}), \forall v\in \mathcal{P}_{-}({\bf V}_i).
\end{equation}
Furthermore, if $u\in W^{k+3,\infty}$, then
\begin{equation}\label{correctionestimatew22}
  \|w_2\|_{0,{\bf V}_i}+\|\partial_tw_2\|_{0,{\bf V}_i}\lesssim \frac{h_i^{k+\frac{5}{2}}}{\|\alpha\|_{0,\infty,{\bf V}_i}}\| u\|_{k+3,\infty,{\bf V}_i}.
\end{equation}
Consequently,
\begin{equation}\label{superconvergencecorrectionw23}
  \tilde{\alpha}_i^2b_{i}(\bar w_0+w_2,v^*)\lesssim h_i^{k+\frac{5}{2}} \| u\|_{k+3,\infty,{\bf V}_i}\|v\|_{{\alpha}_i^2,\mathbf{V}_i},\forall v\in \mathcal{P}_{-}({\bf V}_i).
\end{equation}
\end{lmm}

\subsection{Superconvergence for the SV flux function approximation}

We   construct a global correction function $w\in V_h$ over the whole domain as follows.
\begin{eqnarray}\label{globalcorrectionfunctionw}
w=\left \{
\begin{array}{ll}
w_1,   &if~\alpha|_{i-\frac{1}{2}}\geq 0,\alpha|_{i+\frac{1}{2}}> 0,\\
w_2,   &if~\alpha|_{i-\frac{1}{2}}\leq 0,\alpha|_{i+\frac{1}{2}}< 0, \\
~0,     &otherwise.\\
\end{array}
\right.
\end{eqnarray}
  Here $w_1$ and $w_2$ are  defined  by \eqref{local correction w1} and \eqref{local correction w2}, respectively.

With this definition, we obtain the following superconvergence properties of the correction function $w$.
\begin{thm}\label{PropertyForCorrectionFw}
Let $u(\cdot,t)\in W^{k+3,\infty}$ be the solution of \eqref{eqn:para}, and $w$ be the correction function defined by \eqref{globalcorrectionfunctionw}.  Then
\begin{equation}\label{PropertyForCorrectionFw1}
\widehat{w}|_{i+\frac{1}{2}}=0,~i\in \mathbb{Z}_{N},\ \  \|w\|_{\alpha}+\|\partial_tw\|_{\alpha} \lesssim h^{k+2} \|u\|_{k+3,\infty},
\end{equation}
\begin{equation}\label{PropertyForCorrectionFw3}
  A_h(u-I_hu+w,v_h^*) \lesssim h^{k+2} \|u\|_{k+3,\infty}\|v_h\|_{\alpha},~\forall v_h\in V_h.
\end{equation}
\end{thm}
\begin{proof}
First,
 \eqref{PropertyForCorrectionFw1}  follows directly from  \eqref{local correction w1}, \eqref{local correction w2},
 \eqref{correctionestimatew1} and \eqref{correctionestimatew22}.
  To  show \eqref{PropertyForCorrectionFw3},    it is sufficient to prove that \eqref{PropertyForCorrectionFw3}
 holds true for all $ v_h\in \mathcal{P}_{-}(\mathbf{V}_i)$ and for all  $ v_h\in \mathcal{P}_{0}(\mathbf{V}_i)$ separately.

Recalling the definition of $A_h(\cdot,\cdot)$ and using  \eqref{AForm}, we have
\begin{equation}\label{SuperPropertyForCorrectionAw}
  A_h(u-I_hu+w,v_h^*)=\sum_{\mathbf{V}_i\in\Omega_1\cup\Omega_2}\tilde{\alpha}_i^2b_i(u-I_hu+w,v^*)+\sum_{\mathbf{V}_i\in\Omega_3}\tilde{\alpha}_i^2b_{i}(u-I_hu,v^*),\forall v_h\in \mathcal{P}_{-}(\Omega),
\end{equation}
For the first term on the right side of \eqref{SuperPropertyForCorrectionAw}, we use  \eqref{superconvergencecorrection1} and \eqref{superconvergencecorrectionw23} and then get
 \begin{equation}\label{firsttermSuperproperty1}
   \sum_{\mathbf{V}_i\in\Omega_1\bigcup\Omega_2}\tilde{\alpha}_i^2b_i(u-I_hu+w,v^*)\lesssim h^{k+2} \|u\|_{k+3,\infty}\|v\|_{\alpha}.
 \end{equation}
For any $\mathbf{V}_i\in \Omega_3$, since $\alpha(x_{i-\frac{1}{2}})\alpha(x_{i-\frac{1}{2}})\leq 0$, then there exists at least one point $\theta_i\in \mathbf{V}_i$ such that $\alpha(\theta_i)=0$,  and thus,
\begin{equation}\label{xishuproperty1}
  \|\alpha\|_{0,\infty,\Omega_3}=\|\alpha'(\xi_i)(x-\theta_i)\|_{0,\infty,\Omega_3}\leq h\|\alpha\|_{1,\infty,\Omega_3}\lesssim h.
\end{equation}
Then the combination of \eqref{124}, \eqref{xishuproperty1} and Cauchy-Schwartz inequality leads to
\begin{eqnarray}\nonumber
                \Big|\sum_{\mathbf{V}_i\in\Omega_3}\tilde{\alpha}_i^2a_{h,i}(u-I_hu,v^*) \Big| &\lesssim& h^{k+1}\sum_{\mathbf{V}_i\in\Omega_3}\tilde{\alpha}_i^2(\|u\|_{k+2,\mathbf{V}_i}+h_i^{\frac{1}{2}}\|u\|_{k+1,\infty})\|v\|_{0,\mathbf{V}_i} \\\label{secondtermSuperproperty1}
                &\lesssim & h^{k+2}(\|u\|_{k+2}+ \|u\|_{k+1,\infty})\|v\|_\alpha,~\forall v\in V_h.
\end{eqnarray}
Then \eqref{PropertyForCorrectionFw3} follows for any $ v_h\in \mathcal{P}_{-}(\mathbf{V}_i)$ from \eqref{SuperPropertyForCorrectionAw}, \eqref{firsttermSuperproperty1} and \eqref{secondtermSuperproperty1}.

For  all $ v\in \mathcal{P}_{0}(\mathbf{V}_i)$, a direct calculation from \eqref{BForm}-\eqref{AForm} and \eqref{vv} yields that
\begin{equation}\label{SuperPropertyForP0}
  A_h(u-I_hu+w,v^*)=\sum_{\mathbf{V}_i\in\Omega_1\cup\Omega_2}\tilde{\alpha}_i^2\Big[( u_t-I_hu_t,v)_i+(\partial_tw,v)_i\Big]+\sum_{\mathbf{V}_i\in\Omega_3}\tilde{\alpha}_i^2(u_t-I_hu_t,v)_i
\end{equation}
For any $\mathbf{V}_i\in\Omega_1\cup\Omega_2$,  the interpolation points are either Gauss points or right/left Radau points.
 By utilizing the Newton-interpolating-remainder representation, we  get
\begin{equation*}
  u-I_hu=u[x_{i,1},\ldots,x_{i,k+1},x]\prod_{j=1}^{k+1}(x-x_{i,j}):=\widetilde{u}(x)\widetilde{L}_{i,k+1}(x),
\end{equation*}
where $\widetilde{u}(x)=u[x_{i,1},\ldots,x_{i,k+1},x]$ and $\widetilde{L}_{i,k+1}(x):=\prod_{j=1}^{k+1}(x-x_{i,j})$.
Using the properties of Legendre polynomials or  Radau polynomials, we have
\[
   (\widetilde{L}_{i,k+1},1)_i=0,
\]
  and thus
\begin{equation}\label{SuperPropertyForP0V1}
  \Big|(u_t-I_hu_t,v)_i\Big|=\Big|\int_{x_{i-\frac{1}{2}}}^{x_{i+\frac{1}{2}}}(\widetilde{u}-\widetilde{u}_0) \widetilde{L}_{i,k+1}(x)v\text{d}x\Big|\lesssim h_i^{k+2}\|u\|_{k+2,\mathbf{V}_i}\|v\|_{0,\mathbf{V}_i}.
\end{equation}
  Here $\widetilde{u}_0$ denotes the cell average of $\widetilde{u}$ in $\mathbf{V}_i$.   Combing \eqref{SuperPropertyForP0}, \eqref{SuperPropertyForP0V1} and \eqref{secondtermSuperproperty1}, we conclude that
  the \eqref{PropertyForCorrectionFw3}  is also valid for all $ v_h\in \mathcal{P}_{0}(\mathbf{V}_i)$.
  The proof is completed.
\end{proof}

\begin{coro}\label{coro:1}
Let $u(\cdot,t)\in W^{k+3,\infty}$ be the solution of \eqref{eqn:para},
and $u_h$ be the SV  solution  of \eqref{SVBFscheme}
with initial solution $u_h(\cdot,0)=I_hu_0$ with $I_hu$  the interpolation function of $u$ defined in \eqref{Newinterpolation}. Then   for both the RSV and LSV methods,
\begin{equation}\label{FluxUhUISuperProperty}
  \|(u_h-I_hu)(\cdot,t)\|_{\alpha}\lesssim h^{k+2}\|u\|_{k+3,\infty}.
\end{equation}
\end{coro}
\begin{proof}
 Let $u_I=I_hu-w$.
For the RSV scheme, by choosing $v_h=u_h-{u}_I$ in \eqref{new stablity for RSV} and using \eqref{PropertyForCorrectionFw3},  we have
\[
  \frac{d}{dt}\|u_h- {u}_I\|_{\alpha}^2\lesssim |A(u- {u}_I,(u_h-{u}_I)^*)|+\|u_h-{u}_I\|_\alpha^2\lesssim h^{2(k+2)}\|u\|^2_{k+3,\infty}+\|u_h-{u}_I\|_{\alpha}^2.
\]
Due to the special choice of  the initial solution and
  the Gronwall inequality,   we get
\begin{equation*}
  \|u_h(\cdot,t)- {u}_I(\cdot,t)\|_{\alpha}\lesssim \|w(\cdot,0)\|_{\alpha}+h^{k+2}\|u\|_{k+3,\infty},
\end{equation*}
  which yields, together with the triangle  inequality and \eqref{PropertyForCorrectionFw1}, that
\begin{equation*}
  \|(u_h-I_hu)(\cdot,t)\|_{\alpha}\lesssim \|w(\cdot,t)\|_{\alpha}+\|u_h(\cdot,t)- {u}_I(\cdot,t)\|_{\alpha}\lesssim h^{k+2}\|u\|_{k+3,\infty}.
\end{equation*}
    Taking $v_h=u_h-{u}_I$ in \eqref{new stablity for LSV}  and using the equivalence between the $\|\cdot\|$ and $|\|\cdot|\|$ norm, we can prove that the same result  still
     holds true for LSV methods.
\end{proof}

  Now we are ready to present the superconvergence results for the flux function approximation.

\begin{thm}\label{SVFluxsuperconvergenceProperty}
Let $u(\cdot,t)\in W^{k+3,\infty}$ be the solution of \eqref{eqn:para}, and $u_h$ be the SV  solution obtained by \eqref{SVBFscheme} with initial solution $u_h(\cdot,0)=I_hu_0$.
 Then for both the RSV and LSV methods, the flux function $\alpha u_h$ has the following superconvergence property:
\begin{equation}\label{SVFluxsuperconvergenceProperty1}
e_f: =\|\alpha u_h-\alpha u_I\|_0\lesssim h^{k+2}\|u\|_{k+3,\infty}.
\end{equation}
correspondingly,
\begin{equation}\label{SVFluxsuperconvergenceAtsomeSpecialPoint}
  e_{f,n}+e_{f,c}+e_{f,r}\lesssim h^{k+2}\|u\|_{k+3,\infty},\ \ e_{f,l}\lesssim h^{k+1}\|u\|_{k+3,\infty},
\end{equation}
where
\begin{eqnarray*}
  e_{f,n}^2 &:=& \frac{1}{N}\sum_{i=1}^N(\alpha u-\alpha \widehat{u}_h)^2(x_{i+\frac{1}{2}}),~~~~e_{f,c}^2:=\frac{1}{N}\sum_{i=1}^N\left(\frac{1}{h_i}\int_{\mathbf{V}_i} (\alpha u-\alpha u_h)\text{d}x\right)^2,\\
  e_{f,r}^2 &:=& \frac{1}{N}\sum_{i=1}^N\sum_{j=1}^{k+1}(\alpha u-\alpha u_h)^2(y_{i,j}),~~e_{f,l}^2:=\frac{1}{N}\sum_{i=1}^N\sum_{j=1}^{k}(\alpha \partial_xu-\alpha \partial_xu_h)^2(z_{i,j}),
\end{eqnarray*}
where $\{y_{i,j}\}_{j=1}^{k+1}$ are the interpolation nodes of the interpolation operator $I_h$ defined in \eqref{Newinterpolation}, and $\{z_{i,j}\}_{j=1}^{k}$ are the roots of $\partial_x (\prod_{j=1}^{k+1}(x-y_{i,j}))$.
\end{thm}
\begin{proof}
Firstly, a direct calculation yields that
\begin{eqnarray}\nonumber
  \|\alpha u_h-\alpha u_I\|_{0,\mathbf{V}_i}^2
   &\lesssim& \int_{\mathbf{V}_i}(\alpha-\alpha_{i-\frac{1}{2}})^2 (u_h-u_I)^2\text{d}x +\int_{\mathbf{V}_i}(\alpha-\alpha_{i+\frac{1}{2}})^2 (u_h-u_I)^2\text{d}x+\|u_h-u_I\|_{\alpha,\mathbf{V}_i}^2\\\nonumber
   &\lesssim& h_i^2\|u_h-u_I\|_{0,\mathbf{V}_i}+\|u_h-u_I\|_{\alpha,\mathbf{V}_i}^2
\end{eqnarray}
The estimate \eqref{SVFluxsuperconvergenceProperty1} follows immediately by  using  Corollary  \ref{coro:1} and Theorem  \ref{ConvegenceForRRSV}.

Secondly, we show the first inequality of \eqref{SVFluxsuperconvergenceAtsomeSpecialPoint}.
On the one hand, by applying \eqref{fluxproperty}  and the  inverse inequality,  we have
\begin{eqnarray*}\nonumber
  |e_{f,n}|^2&=&\frac{1}{N}\sum_{i=1}^N\alpha^2(x_{i+\frac{1}{2}})(I_hu-\widehat{u}_h)^2(x_{i+\frac{1}{2}})\\
  &\lesssim & \sum_{i=1}^{N}\left(h^2\|I_hu-u_h\|_{0,\mathbf{V}_i}^2+\|I_hu-u_h\|_{\alpha,\mathbf{V}_i}^2\right).
\end{eqnarray*}
On the other hand, by choosing $v=1$ in \eqref{SuperPropertyForP0V1}, we have  for all  $\mathbf{V}_i\in\Omega_1\cup\Omega_2$,
\begin{equation}\label{FluxuhuIsuperconvergenceProperty1}
  \int_{\mathbf{V}_i}\alpha (u-I_hu)\text{d}x=\int_{\mathbf{V}_i}(\alpha-\alpha_{i+\frac 12}) (u-I_hu)\text{d}x+\alpha_{i+\frac 12}\int_{\mathbf{V}_i} (u-I_hu)\text{d}x\lesssim h_i^{k+3}\|u\|_{k+3,\mathbf{V}_i,\infty}.
\end{equation}
 As for  $\mathbf{V}_i\in\Omega_3$,  we have  $\alpha(x_{i-\frac{1}{2}})\alpha(x_{i-\frac{1}{2}})\leq 0$ in $\Omega_3$.  Then here exists at least one point $\theta_i\in \mathbf{V}_i$ such that $\alpha(\theta_i)=0$, indicating that \eqref{FluxuhuIsuperconvergenceProperty1} still holds in $\mathbf{V}_i\in\Omega_3$. Therefore,
\begin{equation*}
  \frac{1}{N}\sum_{i=1}^N\left(\frac{1}{h_i}\int_{\mathbf{V}_i}\alpha (u-I_hu)\text{d}x\right)^2\lesssim \frac{1}{N}\sum_{i=1}^Nh_i^{2k+4}\|u\|_{k+3,\mathbf{V}_i,\infty}^2\lesssim h^{2k+4}\|u\|_{k+3,\infty}^2.
\end{equation*}
By applying \eqref{SVFluxsuperconvergenceProperty1} and Cauchy-Schwarz inequality, we derive that
\begin{equation*}
  \frac{1}{N}\sum_{i=1}^N\left(\frac{1}{h_i}\int_{\mathbf{V}_i}\alpha (I_hu-u_h)\text{d}x\right)^2\lesssim  \|I_hu-u_h\|^2_{\alpha}\lesssim h^{2k+4}\|u\|_{k+3,\infty}^2.
  \end{equation*}
Thus, the application of triangle inequality yields that
\begin{equation*}
  e_{f,c}^2\le \frac{1}{N}\sum_{i=1}^N\left(\frac{1}{h_i}\int_{\mathbf{V}_i}\alpha (I_hu-u)\text{d}x\right)^2+ \frac{1}{N}\sum_{i=1}^N\left(\frac{1}{h_i}\int_{\mathbf{V}_i}\alpha (I_hu-u_h)\text{d}x\right)^2  \lesssim h^{2k+4}\|u\|_{k+3,\infty}^2.
\end{equation*}
 We next  estimate the error  $e_{f,r}$.  Using the inverse inequality, we get
\begin{eqnarray*}
 e_{f,r}= \frac{1}{N} \sum_{i=1}^N\sum_{j=1}^{k+1}\alpha^2(y_{i,j}) (I_hu-u_h)^2(y_{i,j})& \lesssim & \sum_{i=1}^N\sum_{j=1}^{k+1}  \alpha^2(x_{i,j}) \|I_hu-u_h\|_{0,\mathbf{V}_i}^2 \\
   &\lesssim&  \sum_{i=1}^N \left(h_i^2 \|I_hu-u_h\|_{0,\mathbf{V}_i}^2+\|I_hu-u_h\|_{\alpha,\mathbf{V}_i}^2\right).
\end{eqnarray*}
Therefore the  first inequality of \eqref{SVFluxsuperconvergenceAtsomeSpecialPoint} follows from the conclusions in Corollary  \ref{coro:1} and Theorem  \ref{ConvegenceForRRSV}.

Finally, by using the Newton-interpolating-remainder representation, we obtain that
\begin{equation*}
  u(x)-I_hu(x)=u[y_{i,1},\ldots,y_{i,k},y_{i,k+1},x]\prod_{j=1}^{k+1}(x-y_{i,j}):=\widetilde{u}(x)\omega(x),
\end{equation*}
which indicates that at the roots of $\omega(x)$,   there holds
\begin{equation*}
  |\partial_x(u-I_hu)(z_{i,j},t)|\lesssim h_i^{k+1}\|\partial_x\widetilde{u}(x)\|_{0,\mathbf{V}_i,\infty}\lesssim h^{k+1}\|u\|_{k+2,\infty}.
\end{equation*}
In addition, by using \eqref{FluxUhUISuperProperty} and the inverse inequality, we get
\begin{eqnarray*}\nonumber
  \frac{1}{N}\sum_{i=1}^N\sum_{j=1}^k\alpha^2(z_{i,j})\partial_x(u_h-I_hu)^2(z_{i,j},t)&\lesssim& \sum_{i=1}^N\sum_{j=1}^k\alpha^2(z_{i,j}) \|\partial_x(u_h-I_hu)\|_{0,\mathbf{V}_i}^2\\
  &\lesssim& h^{-2}\|u_h-I_hu\|_{\alpha}^2 +\|u_h-I_hu\|_{0}^2.
\end{eqnarray*}
 Therefore,
\begin{equation*}
   e_{f,l}^2=\frac{1}{N}\sum_{i=1}^N\sum_{j=1}^k\alpha^2(z_{i,j})\partial_x(u-I_hu+I_hu-u_h)^2(z_{i,j},t)\lesssim h^{2k+2}\|u\|_{k+3,\infty}^2.
\end{equation*}
The proof is completed.
\end{proof}
\subsection{Superconvergence for the SV solution approximation}

 In this subsection, we establish the superconvergence properties for solution approximation of the LSV and RSV methods.

 Assume that the smooth function $\alpha(x)$ has only a finite number of zeros
on $\Omega$. For convenience, we suppose that $\alpha$ has only one zero point at $x=0$ and there exists a positive integer $m$ such that
\begin{equation}\label{sufficientfunctioncondition}
  \alpha(0)=\alpha'(0)=\cdots=\alpha^{(m-1)}(0)=0,\alpha^{m}(0)\neq 0.
\end{equation}
Following the similar idea of \cite{Cao-Shu-ZhangM2AN}, we  slightly  modify  the correction functions $w_1$ and $w_2$ defined in \eqref{local correction w1} and \eqref{local correction w2}, i.e.,
\begin{eqnarray}\label{modifiedcorrectionfunctionw}
\widetilde{w}_i|_{\mathbf{V}_j}=\left \{
\begin{array}{ll}
0,   &\mathbf{V}_j\subset \Lambda=[0,x_{i_0+\frac{1}{2}}],\\
w_i,   &\mathbf{V}_j\subset \Lambda^+=\Omega\setminus\Lambda, \\
\end{array}
\right.
\end{eqnarray}
where $i\in\mathbb{Z}_2$, the positive integer $i_0$ satisfies that $x_{i_0-\frac{1}{2}}\leq h^{\frac{1}{m'}}\leq x_{i_0+\frac{1}{2}}$, and $m'=\min\{m,k+3\}$. Then we  define the global correction function $\widetilde{w}$  by
\begin{eqnarray}\label{globalmodifiedcorrectionfunction}
\widetilde{w}|_{\mathbf{V}_i}=\left \{
\begin{array}{ll}
\widetilde{w}_1,   &if~\alpha|_{i-\frac{1}{2}}\geq 0,\alpha|_{i+\frac{1}{2}}> 0,\\
\widetilde{w}_2,   &if~\alpha|_{i-\frac{1}{2}}\leq 0,\alpha|_{i+\frac{1}{2}}< 0, \\
0,     &otherwise.\\
\end{array}
\right.
\end{eqnarray}

\begin{lmm}Let $u\in W^{k+3,\infty}$ be the solution of \eqref{eqn:para}, $\alpha(x)$ be a sufficiently smooth function satisfying \eqref{sufficientfunctioncondition}, and $\widetilde{w}_i(i\in\mathbb{Z}_2)$ be the modified correction functions defined by \eqref{modifiedcorrectionfunctionw}. Then
\begin{equation}\label{modifiedcorrectionfunctionsproperty}
  \sum_{i=1}^2\left(\|\widetilde{w}_i\|_0+\|\partial_t\widetilde{w}_i\|_0\right)\lesssim h^{k+1+\frac{1}{2m'}}\|u\|_{k+3,\infty}.
\end{equation}
Furthermore,  it holds that
\begin{equation}\label{bmodifiedcorrectionfunctionsproperty}
  b(u-I_hu+\widetilde{w},v_h^*)\lesssim h^{k+1+\frac{1}{2m'}}\|u\|_{k+3,\infty}\|v_h\|_0,~v_h\in V_h.
\end{equation}
\end{lmm}
\begin{proof}Here we omit the proof of \eqref{modifiedcorrectionfunctionsproperty} and refer to \cite{Cao-Shu-ZhangM2AN}
for more detailed information and discussions. We focus our attention to prove \eqref{bmodifiedcorrectionfunctionsproperty}.
Replacing  $w_1$ in  \eqref{Correctionproperty} by $\widetilde{w}_1$, and following the same argument as what we did in \eqref{Correctionproperty} yields
\begin{eqnarray*}\nonumber
 | b(u-I_h^-u+\widetilde{w}_1,v_h^*)|&=& |\sum_{\mathbf{V}_i\in \Lambda}(\partial_t(u-I_h^-u),v_h^*)_i+\sum_{\mathbf{V}_i\in \Lambda^+}(\partial_tw_1,v_h^*)_i-\sum_{\mathbf{V}_i\in \Lambda^+}\sum_{j=1}^k(\alpha(x_{i,j})-\overline{\alpha}_i)v_x(x_{i,j})w_1(x_{i,j})|\\ \label{123404321}
  &\lesssim& h^{k+1}\|u\|_{k+2,\infty}x_{i_0+\frac{1}{2}}^{\frac{1}{2}}\|v_h\|_{0,\Lambda}+\left(\|\widetilde{w}_1\|_0+\|\partial_t\widetilde{w}_1\|_0\right)\|v_h\|_0\\
 &\lesssim&  h^{k+1+\frac{1}{2m'}} \|u\|_{k+3,\infty}\|v_h\|_0,\ \ \ \forall v_h\in \mathcal{P}_{-}({\bf V}_i).
\end{eqnarray*}
 Here in the last step,  we have used  $x_{i_0+\frac{1}{2}}\lesssim h^{\frac{1}{m'}}$ and  \eqref{modifiedcorrectionfunctionsproperty}.
As for $v_h\in \mathcal{P}_{0}(\mathbf{V}_i)$, we have from
 \eqref{SuperPropertyForP0V1} and \eqref{modifiedcorrectionfunctionsproperty} that
\begin{equation*}
  |b(u-I_h^-u+\widetilde{w}_1,v_h^*)|=\sum_{\mathbf{V}_i\in \Lambda}(\partial_t(u-I_h^-u),v_h)_i+\sum_{\mathbf{V}_i\in \Lambda^+}(\partial_tw_1,v_h)_i\lesssim h^{k+1+\frac{1}{2m'}} \|u\|_{k+3,\infty}\|v_h\|_0.
\end{equation*}
Therefore,  for all $v_h\in V_h$, $v_h=v_0+v_1$, where $v_0\in \mathcal{P}_{0}(\mathbf{V}_i)$ and $v_1\in \mathcal{P}_{-}({\bf V}_i)$, we have
\begin{equation}\label{superconvergencepertyPart1}
 | b(u-I_h^-u+\widetilde{w}_1,v_h^*)|=|b(u-I_h^-u+\widetilde{w}_1,v_0^*)+b(u-I_h^-u+\widetilde{w}_1,v_1^*)|\lesssim h^{k+1+\frac{1}{2m'}} \|u\|_{k+3,\infty}\|v_h\|_0.
\end{equation}
Likewise, we get
 \begin{equation}\label{superconvergencepertyPart2}
  |b(u-I_h^+u+\widetilde{w}_2,v_h^*)|\lesssim h^{k+1+\frac{1}{2m'}} \|u\|_{k+3,\infty}\|v_h\|_0,\forall v_h\in V_h.
\end{equation}
Noticing that
 \begin{equation}\label{superconvergencepertyPart3}
  b(u-I_hu+\widetilde{w},v_h^*)=\sum_{\mathbf{V}_i\in \Omega_1}b_i(u-I_h^-u+\widetilde{w}_1,v_h^*)+\sum_{\mathbf{V}_i\in \Omega_2}b_i(u-I_h^+u+\widetilde{w}_2,v_h^*)+\sum_{\mathbf{V}_i\in \Omega_3}b_i(u-I_h^{\pm}u,v_h^*).
\end{equation}
By using the fact that $|\Omega_3|\lesssim h$, the inequalities \eqref{xishuproperty1} and \eqref{Interproperty1}, we obtain
\begin{eqnarray}\nonumber
  |\sum_{\mathbf{V}_i\in \Omega_3}b_i(u-I_h^{\pm}u,v_h^*)|&=& |\sum_{\mathbf{V}_i\in \Omega_3}(\partial_t(u-u_I),v_h^*)_i-\sum_{\mathbf{V}_i\in \Omega_3}\sum_{j=1}^kA_{i,j}\alpha(x_{i,j})(u-u_I)_x(x_{i,j})v_h(x_{i,j})|\\ \label{123454321}
  &\lesssim& h^{k+1+\frac{1}{2m'}} \|u\|_{k+3,\infty}\|v_h\|_0,\forall v_h\in V_h.
\end{eqnarray}
Then the inequality \eqref{bmodifiedcorrectionfunctionsproperty} follows from \eqref{superconvergencepertyPart1}, \eqref{superconvergencepertyPart2}, and \eqref{superconvergencepertyPart3}, and \eqref{123454321}.
\end{proof}
\begin{thm}\label{SVSolutionsuperconvergenceProperty}
Let $u\in W^{k+3,\infty}$ be the solution of \eqref{eqn:para}, and $u_h$ be the solution of \eqref{SVBFscheme} with the initial solution $u_h^0=I_hu_0$, where $I_h$ is defined by \eqref{Newinterpolation}. Suppose $\alpha(x)$ is a sufficiently smooth function satisfying \eqref{sufficientfunctioncondition}. Then for $m'=\min\{m,k+3\}$
\begin{equation}\label{SVSolutionIhsuperconvergenceProperty}
  e_u:=\|u_h-I_hu\|_0\lesssim h^{k+1+\frac{1}{2m'}} \|u\|_{k+3,\infty},
\end{equation}
  Moreover, there hold
  \begin{equation}\label{SVSolutionsuperconvergenceAtsomeSpecialPoint}
  e_{u,n}+e_{u,c}+e_{u,r}\lesssim h^{k+1+\frac{1}{2m'}}\|u\|_{k+3,\infty},\ \  e_{u,l}\lesssim h^{k+\frac{1}{2m'}}\|u\|_{k+3,\infty},
\end{equation}
where
\begin{eqnarray*}
  e_{u,n}^2 &:=& \frac{1}{N}\sum_{i=1}^N( u-\widehat{u}_h)^2(x_{i+\frac{1}{2}}),~~~~e_{u,c}^2:=\frac{1}{N}\sum_{i=1}^N\left(\frac{1}{h_i}\int_{\mathbf{V}_i} (u-u_h)\text{d}x\right)^2,\\
  e_{u,r}^2 &:=& \frac{1}{N}\sum_{i=1}^N\sum_{j=1}^{k+1}(u-u_h)^2(y_{i,j}),~~e_{u,l}^2:=\frac{1}{N}\sum_{i=1}^N\sum_{j=1}^{k}(u-u_h)^2(z_{i,j}),
\end{eqnarray*}
where the nodes $\{y_{i,j}\}_{j=1}^{k+1}$ and $\{z_{i,j}\}_{j=1}^{k}$ are the same as in Theorem 5.2.
\end{thm}
\begin{proof}
By the definition of $\widetilde{w}$ in \eqref{globalmodifiedcorrectionfunction} and the property \eqref{PropertyForCorrectionFw1}, it is easy to verify that the numerical flux of $\widetilde{w}$ satisfies $\widehat{\widetilde{w}}|_{i+\frac{1}{2}}=0$ for all $i\in \mathbb{Z}_{N}.$ Consequently,
choosing $v=u_h-I_hu+\widetilde{w}$ in \eqref{RRSVbasionAll} and applying the property \eqref{fluxproperty} and \eqref{bmodifiedcorrectionfunctionsproperty}, we obtain that
\begin{eqnarray*}\nonumber
  \frac{d}{dt}\|u_h-I_hu+\widetilde{w}\|_0^2&\lesssim & \|u_h-I_hu+\widetilde{w}\|_0^2+|a_{h}(u-I_hu+\widetilde{w},(u_h-I_hu+\widetilde{w})^*)|\\ \nonumber
  &=& \|u_h-I_hu+\widetilde{w}\|_0^2+|b(u-I_hu+\widetilde{w},u_h^*-u_I^*+\widetilde{w}^*)|\\ \label{12345654321}
  &\lesssim&\|u_h-u_I+\widetilde{w}\|_0^2+h^{k+1+\frac{1}{2m'}} \|u\|_{k+3,\infty}\|u_h-u_I+\widetilde{w}\|_0.
\end{eqnarray*}
  Then \eqref{SVSolutionIhsuperconvergenceProperty} follows from the Gronwall inequality and \eqref{modifiedcorrectionfunctionsproperty}.
  The proof of \eqref{SVSolutionsuperconvergenceAtsomeSpecialPoint} is similar to that of  Theorem \ref{SVFluxsuperconvergenceProperty} and
  we omit it here.
\end{proof}

\section{Numerical Results}In this section, we present some numerical experiments to verify our theoretical findings.
We operate our programs in MATLAB 2019b. Uniform meshes of $n$ elements are used in our numerical experiments. We use the fourth-order Runge-Kutta method with time step $\Delta t=0.01\frac{1}{n}$ to reduce the time discretization.

{\it Example} 1. We consider the following equation with the periodic boundary condition:
\[
    u_t+(\sin(x)u)_x=g(x,t),\ \ (x,t)\in [0,2\pi]\times(0,\frac{\pi}{2}],\ \ u(x,0)=\sin(x).
\]
We choose $g(x,t)$ such that the exact solution to this problem is
$
   u(x,t)=\text{e}^{\sin(x-t)}.
$
  Note that $\alpha(x)=\sin(x)$ has three zeros $x=0,\pi, 2\pi$, and at these zeros,
 \[
     \partial_x\alpha(x)=\cos(x)\neq 0.
 \]
 Then \eqref{sufficientfunctioncondition} holds true with $m=1$, which indicates $m'=1$ in Theorem 5.3.

The problem is solved by the LSV and RSV schemes with $k=1,2,3$.
 Listed in Table 6.1 are
 the errors of $u-u_h$ in $L^2$ and  $L^{\infty}$  norms and the corresponding convergence rates,
 where the optimal convergence rate of $h^{k+1}$ for both LSV and RSV  methods are observed. These results confirm our theoretical results
 established in Theorem \ref{ConvegenceForRRSV}.

We also test the superconvergence phenomenon of the flux function approximation and solution itself approximation for both RSV and LSV,
and measure several errors between the numerical solution and the exact solution, which are defined in Theorems 5.2-5.3.
We list in Tables 6.2-6.3  and Tables 6.4-6.5 the numerical errors and the corresponding convergence rates for RSV methods and LSV methods, respectively.
For both   RSV and LSV methods,
we observe  a convergence rate of $k+2$ for the errors $e_{f,c}, e_{f,n}, e_{f}, e_{f,r}$ and a rate of $k+1$ for $e_{f,l}$, which are consistent with the theoretical findings in Theorem 5.2 and Corollary 5.2. Note that when $k=1$, there exists no superconvergence  phenomenon of the function value approximation  for the LSV method.
As for the solution itself approximation, we see  a convergence rate of  $(k+\frac{3}{2})$ for  $e_u$, $e_c$,  $e_{r}$ and $e_n$, and a rate of  $(k+\frac{1}{2})$ for $e_{u,l}$.
These results are consistent with the conclusions in Theorem 5.3.


\begin{table}[!h]
	\small{\caption{\emph{Errors and convergence rates of RSV and LSV methods for  Example 1. }}}
	\label{Tab:Eg1a}	\centering
	\begin{tabular}{ll|llll|llll}
		\hline
		$~ $  &       &\multicolumn{4}{c}{RSV}& \multicolumn{4}{c}{LSV} \\
		\hline
		$~~k$&$n $ &$\|u-u_h\|_0$   &order  &$\|u-u_h\|_{0,\infty}$  &order  &$\|u-u_h\|_0$   &order  &$\|u-u_h\|_{0,\infty}$  &order\\
		\hline
		\multirow{4}{*}{$k=1$}
		& $~128$ & 4.24E-4 &      &1.10E-3  &       & 6.31E-4 &      &1.60E-3  &      \\
		& $~256$ & 1.06E-4 &2.00  &2.73E-4 &$2.00$ & 1.58E-4 &2.00  &4.09E-4 &$2.00$\\
		& $~512$ & 2.64E-5 &2.00  &6.82E-5 &$2.00$ & 3.95E-5 &2.00  &1.02E-4 &$2.00$\\
		& $1024$ & 6.59E-6 &2.00  &1.71E-5 &$2.00$ & 9.87E-6 &2.00  &2.56E-5 &$2.00$\\
		\hline
		\multirow{4}{*}{$k=2$}
		& $~128$ & 2.81E-6 &      &8.02E-6 &       &4.42E-6 &      &1.34E-5 &              \\
		& $~256$ & 3.51E-7 &3.00  &1.00E-6 &$3.00$ &5.52E-7 &3.00  &1.67E-6 &$3.00$\\
		& $~512$ & 4.39E-8 &3.00  &1.25E-7 &$3.00$ &6.90E-8 &3.00  &2.09E-7 &$3.00$ \\
		& $1024$ & 5.49E-9 &3.00  &1.57E-8 &$3.00$ &8.63E-9 &3.00  &2.61E-8 &$3.00$\\
		\hline
		\multirow{4}{*}{$k=3$}
		& $~~32$ & 4.82E-6 &      &1.84E-5 &       &7.66E-6 &      &3.27E-5 &      \\
		& $~~64$ & 2.97E-7 &4.02  &1.19E-6 &$3.95$ &4.75E-7 &4.01  &2.09E-6 &$3.97$\\
		& $~128$ & 1.84E-8 &4.01  &7.50E-8 &$3.99$ &2.96E-8 &4.00  &1.31E-7 &$3.99$\\
		& $~256$ & 1.15E-9 &4.01  &4.70E-9 &$4.00$ &1.85E-9 &4.00  &8.22E-9 &$4.00$\\
		\hline
	\end{tabular}
\end{table}

\begin{table}[!h]
	\centering
	\small{\caption{\emph{Errors and   convergence orders for the RSV flux function approximation  for Example 1.}}}
	\label{Tab:Eg1b}
	\begin{tabular}{lllllllllllll}
		\hline
		$~~~~k$&$~~n $ &~~$e_{f}$   &order &~~$e_{f,c}$  &order  &~~$e_{f,r}$&order  &~~$e_{f,n}$  &order &~~$e_{f,l}$  &order\\
		\hline
		\multirow{4}{*}{$k=1$}
		&$~128$  & 2.07E-7 &      &1.60E-6 &     &3.11E-6 &      &3.13E-6 &         &2.35E-4 &   \\
		&$~256$  & 3.02E-8 &2.78  &1.94E-7 &3.04 &3.97E-7 &$2.97$&3.98E-7 &$~2.98$  &5.94E-5 &$~1.98$   \\
		&$~512$  & 4.05E-9 &2.90  &2.38E-8 &3.02 &4.97E-8 &$3.00$&4.98E-8 &$~3.00$  &1.49E-5 &$~2.00$   \\
		&$1024$  & 5.23E-10 &2.95  &2.95E-9 &3.01 &6.22E-9 &$3.00$&6.22E-9 &$~3.00$  &3.71E-6 &$~2.00$   \\
		\hline
		\multirow{4}{*}{$k=2$}
		&$~128$  & 8.96E-10 &      &6.68E-10 &      &2.02E-8  &       &2.02E-8  &         &2.46E-6 &   \\
		&$~256$  & 3.55E-11 &4.66  &2.26E-11 &4.89  &1.25E-9  &$4.02$&1.25E-9  &$~4.02$  &3.00E-7 &$~3.03$   \\
		&$~512$  & 1.48E-12 &4.59  &8.24E-13 &4.78  &7.73E-11 &$4.01$&7.73E-11  &$~4.01$  &3.71E-8 &$~3.01$   \\
		&$1024$  & 6.36E-14 &4.54  &3.49E-14 &4.56  &4.82E-12 &$4.00$&4.82E-12  &$~4.00$  &4.63E-9 &$~3.00$   \\
		\hline
		\multirow{4}{*}{$k=3$}
		&$~~32$  & 9.45E-8 &       &2.78E-8  &      &2.11E-7  &      &1.47E-7  &         &4.92E-6 &   \\
		&$~~64$  &2.17E-9 &5.44  &5.90E-10 &5.56  &5.74E-9  &$5.20$&4.29E-9  &$~5.09$  &3.27E-7 &$~3.91$   \\
		&$~128$  & 4.92E-11 &5.46  &1.38E-11 &5.42  &1.51E-10 &$5.25$&1.22E-10 &$~5.14$  &2.04E-8 &$~4.00$   \\
		&$~256$  & 1.12E-12 &5.45  &3.25E-13 &5.43  &4.08E-12 &$5.21$&3.66E-12 &$~5.06$  &1.26E-9 &$~4.01$   \\
		\hline
	\end{tabular}
\end{table}

\begin{table}[!h]
	\centering
	\small{\caption{\emph{Errors and  convergence rates for the RSV solution itself approximation for Example 1.}}}
	\label{Tab:Eg1c}
	\begin{tabular}{lllllllllllll}
		\hline
		$~~~~k$&$~~n $ &~~$e_{u}$   &order &~~$e_{u,c}$  &order  &~~$e_{u,r}$&order  &~~$e_{u,n}$  &order &~~$e_{u,l}$  &order\\
		\hline
		\multirow{4}{*}{$k=1$}
		&$~128$  & 2.60E-7 &      &2.51E-5 &       &2.51E-5 &      &2.41E-5 &         &5.30E-4 &   \\
		&$~256$  & 3.82E-8 &2.77  &4.43E-6 &$2.50$ &4.39E-6 &$2.59$&4.24E-6 &$~2.51$  &1.66E-5 &$~1.66$   \\
		&$~512$  & 5.36E-9 &2.83  &7.84E-7 &$2.50$ &7.71E-7 &$2.55$&7.45E-7 &$~2.51$  &5.42E-5 &$~1.61$   \\
		&$1024$  & 7.58E-10 &2.82  &1.39E-7 &$2.50$ &1.36E-7 &$2.53$&1.31E-7 &$~2.50$  &1.83E-6 &$~1.56$   \\
		\hline
		\multirow{4}{*}{$k=2$}
		&$~128$  & 1.56E-8  &      &7.55E-9 &       &5.54E-8 &      &5.45E-8 &         &6.20E-6 &   \\
		&$~256$  & 1.59E-9  &3.30  &7.10E-10&$3.41$ &4.63E-9 &$3.58$&4.55E-9 &$~3.58$  &1.01E-6 &$~2.62$   \\
		&$~512$  & 1.50E-10 &3.41  &6.67E-11&$3.42$ &3.97E-10 &$3.54$&3.90E-10 &$~3.54$  &1.71E-7 &$~2.56$   \\
		&$1024$  & 1.36E-11 &3.46  &6.12E-12&$3.45$ &3.46E-11 &$3.52$&3.40E-11&$~3.52$ &2.96E-8 &$~2.53$   \\
		\hline
		\multirow{4}{*}{$k=3$}
		&$~~32$  & 6.05E-7 &       &1.47E-7 &        &9.91E-7 &      &6.56E-7 &          &2.11E-5 &   \\
		&$~~64$  & 2.75E-8 &4.46  &6.55E-9 &$4.49$  &4.70E-8 &$4.40$&2.89E-8 &$~4.51$  &1.78E-6 &$~3.57$   \\
		&$~128$  & 1.23E-9 &4.54  &2.95E-10 &$4.48$  &2.15E-9 &$4.45$&1.28E-9 &$~4.50$  &1.54E-7 &$~3.53$   \\
		&$~256$  & 5.48E-11 &4.52  &1.32E-11 &$4.48$ &9.66E-11 &$4.48$&5.66E-11 &$~4.50$ &1.34E-8 &$~3.52$   \\
		\hline
	\end{tabular}
\end{table}

\begin{table}[!h]
	\centering
	\small{\caption{\emph{Errors and   convergence orders for the LSV flux function approximation  for Example 1.}}}
	\label{Tab:Eg1d}
	\begin{tabular}{lllllllllllll}
		\hline
		$~~~~k$&$~~n $ &~~$e_f$   &order &~~$e_{f,c}$  &order  &~~$e_{f,r}$&order  &~~$e_{f,n}$  &order &~~$e_{f,l}$  &order\\
		\hline
		\multirow{4}{*}{$k=1$}
		&$~128$  & 7.36E-6 &      &9.80E-5 &       &7.77E-6 &      &5.32E-6 &         &3.10E-4 &   \\
		&$~256$  & 1.85E-6 &1.99  &2.45E-5 &$2.00$ &1.99E-6 &$1.96$&1.33E-6 &$~2.00$  &8.16E-5 &$~2.00$   \\
		&$~512$  & 4.64E-7 &2.00  &6.14E-6 &$2.00$ &4.98E-7 &$2.00$&3.33E-7 &$~2.00$  &2.05E-5 &$~2.00$   \\
		&$1024$  & 1.16E-7 &2.00  &1.54E-6 &$2.00$ &1.25E-8 &$2.00$&8.33E-8 &$~2.00$  &5.14E-6 &$~2.00$   \\
		\hline
		\multirow{4}{*}{$k=2$}
		&$~128$  & 1.41E-9 &      &8.49E-9 &       &3.08E-8 &      &3.08E-8 &         &2.58E-6 &   \\
		&$~256$  & 1.14E-10 &3.63  &5.36E-10 &$3.98$ &1.95E-9 &$3.98$&1.94E-9 &$~3.98$  &3.17E-7 &$~3.03$   \\
		&$~512$  & 8.05E-12 &3.82  &3.36E-11 &$3.99$ &1.21E-10 &$4.00$&1.21E-10 &$~4.00$  &3.92E-8 &$~3.02$   \\
		&$1024$  & 5.35E-13 &3.91  &2.09E-12 &$4.01$ &7.61E-12 &$4.00$&7.61E-12 &$~4.00$  &4.88E-9 &$~3.01$   \\
		\hline
		\multirow{4}{*}{$k=3$}
		&$~~32$  & 5.24E-8  &      &6.28E-8  &      &2.71E-7  &      &2.42E-7  &         &6.36E-6 &   \\
		&$~~64$  & 1.13E-9  &5.54  &1.36E-9  &5.53  &7.72E-9  &$5.13$&6.94E-9  &$~5.12$  &4.33E-7 &$~3.88$   \\
		&$~128$  & 2.51E-11 &5.49  &3.03E-11 &5.49  &2.13E-10 &$5.18$&1.97E-10 &$~5.14$  &2.63E-8 &$~4.04$   \\
		&$~256$  & 5.62E-13 &5.48  &6.79E-13 &5.48  &6.24E-12 &$5.10$&5.95E-12 &$~5.05$  &1.61E-9 &$~4.03$   \\
		\hline
	\end{tabular}
\end{table}

\begin{table}[!h]
	\centering
	\small{\caption{\emph{Errors and  convergence rates for the LSV solution itself approximation for Example 1.}}}
	\label{Tab:Eg1e}
	\begin{tabular}{lllllllllllll}
		\hline
		$~~~~k$&$~~n $ &~~$e_u$   &order &~~$e_{u,c}$  &order  &~~$e_{u,r}$&order  &~~$e_{u,n}$  &order &~~$e_{u,l}$  &order\\
		\hline
		\multirow{4}{*}{$k=1$}
		&$~128$  & 9.48E-6 &      &1.17E-4 &       &3.63E-5 &      &4.36E-5 &         &7.71E-4 &   \\
		&$~256$  & 2.38E-6 &1.99  &2.90E-5 &$2.02$ &9.28E-6 &$1.97$&1.02E-5 &$~2.09$  &2.43E-4 &$~1.66$   \\
		&$~512$  & 5.97E-7 &2.00  &7.23E-6 &$2.01$ &2.36E-6 &$1.98$&2.48E-6 &$~2.05$  &7.99E-5 &$~1.60$   \\
		&$1024$  & 1.50E-7 &2.00  &1.80E-6 &$2.00$ &5.95E-7 &$1.99$&6.10E-7 &$~2.02$  &2.70E-5 &$~1.56$   \\
		\hline
		\multirow{4}{*}{$k=2$}
		&$~128$  & 1.70E-8 &      &3.03E-9 &       &5.65E-8 &      &5.64E-8 &         &6.32E-6 &   \\
		&$~256$  & 1.52E-9 &3.49  &2.17E-9 &$3.80$ &4.74E-9 &$3.57$&4.71E-9 &$~3.58$  &1.03E-6 &$~2.62$   \\
		&$~512$  & 1.34E-10&3.50  &1.65E-10 &$3.72$ &3.96E-10 &$3.58$&3.94E-10 &$~3.58$  &1.74E-7 &$~2.57$   \\
		&$1024$  & 1.19E-11&3.50  &1.33E-11 &$3.64$ &3.39E-11 &$3.55$&3.37E-11 &$~3.55$  &3.00E-8 &$~2.53$   \\
		\hline
		\multirow{4}{*}{$k=3$}
		&$~~32$  & 2.25E-7 &      &3.21E-7 &       &1.09E-6 &      &9.36E-7 &         &2.33E-5 &   \\
		&$~~64$  & 9.46E-9 &4.57  &1.45E-8 &$4.47$ &5.01E-8 &$4.44$&3.81E-8 &$~4.62$  &2.08E-6 &$~3.49$   \\
		&$~128$  & 4.06E-10 &4.54  &6.50E-10 &$4.48$ &2.27E-9 &$4.46$&1.61E-9 &$~4.56$  &1.85E-7 &$~3.49$   \\
		&$~256$  & 1.77E-11 &4.52  &2.90E-11 &$4.49$ &1.02E-10 &$4.48$&7.00E-11 &$~4.53$  &1.64E-8 &$~3.49$   \\
		\hline
	\end{tabular}
\end{table}

{\it Example} 2. We consider the following equation with the periodic boundary condition:
\[
    u_t+(\sin^2(x)u)_x=g(x,t),\ \ (x,t)\in [0,2\pi]\times(0,\frac{\pi}{2}],\ \ u(x,0)=\text{e}^{\sin(x)}.
\]
We choose $g(x,t)$ such that the exact solution to this problem is
$
   u(x,t)=\text{e}^{\sin(x-t)}.
$
  In this case,
 \[
     \alpha(x)=\partial_x\alpha(x)=0,\ \ \partial_x^2\alpha(x)\neq 0,\ \ x=0,\pi,2\pi.
 \]
 Then $m=m'=2$ in Theorem 5.3.

We compute the same errors as in Example 1 on the same uniform meshes at time $t=2\pi$.
The computational results for the flux function approximation and for the  solution approximation itself for both LSV and RSV methods
are given in Tables 6.6-6.10.

In Table 6.6, we observe an optimal convergence rate for both the errors $\|u-u_h\|_0$ and $\|u-u_h\|_{0,\infty}$.
Similar to Example 1,  the convergence rates of $e_{f},e_{f,n}, e_{f,r}, e_{f,c}$ are at least $k+2$,  and the convergence rate of $e_{f,l}$ is $k+1$,
which  indicates that  most of the error bounds in Theorem 4.1, Theorem 5.2 and
Corollary 5.2 are sharp.  For the cell  average error $e_c$,  the convergence rate in case $k=2$ for RSV is $5$, one order higher than the theoretical result.
Again, we do not observe superconvergence results in case $k=1$ for LSV methods.

\begin{table}[!h]
\centering
\small{\caption{\emph{Errors and convergence rates of RSV and LSV methods for  Example 2.}}}
\label{Tab:Eg2a}
\begin{tabular}{ll|llll|llll}
\hline
$~ $  &       &\multicolumn{4}{c}{RSV}& \multicolumn{4}{c}{LSV} \\
\hline
 $~~k$&$~~n $ &$\|u-u_h\|_0$   &order  &$\|u-u_h\|_{0,\infty}$  &order  &$\|u-u_h\|_0$   &order  &$\|u-u_h\|_{0,\infty}$  &order\\
\hline
\multirow{4}{*}{$k=1$}
& $~128$ & 3.09E-4 &      &9.86E-4 &       & 8.59E-4 &      &1.52E-3 &      \\
& $~256$ & 9.94E-5 &1.97  &2.58E-4 &$1.94$ & 2.15E-4 &2.00  &3.72E-4 &$2.03$\\
& $~512$ & 2.53E-5 &1.98  &6.60E-5 &$1.97$ & 5.41E-5 &1.99  &9.24E-5 &$2.01$\\
& $1024$ & 6.39E-6 &1.98  &1.67E-5 &$1.98$ & 1.36E-5 &1.99  &2.29E-5 &$2.01$\\
\hline
\multirow{4}{*}{$k=2$}
& $~128$ & 3.06E-6 &      &8.50E-6 &       &4.61E-6 &      &1.33E-5 &              \\
& $~256$ & 3.76E-7 &3.03  &1.08E-6 &$2.97$ &5.73E-7 &3.01  &1.67E-6 &$3.00$\\
& $~512$ & 4.62E-8 &3.02  &1.32E-7 &$3.03$ &7.11E-8 &3.01  &2.09E-7 &$3.00$ \\
& $1024$ & 5.70E-9 &3.02  &1.64E-8 &$3.01$ &8.83E-9 &3.01  &2.61E-8 &$3.00$\\
\hline
\multirow{4}{*}{$k=3$}
& $~~32$ & 5.18E-6 &      &1.82E-5 &       &6.97E-6 &      &2.92E-5 &      \\
& $~~64$ & 3.03E-7 &4.09  &1.11E-6 &$4.04$ &4.32E-7 &4.00  &2.11E-6 &$3.79$\\
& $~128$ & 1.86E-8 &4.03  &7.15E-7 &$3.95$ &2.73E-8 &3.98  &1.32E-7 &$4.00$\\
& $~256$ & 1.15E-9 &4.01  &4.60E-8 &$3.96$ &1.74E-9 &3.97  &8.51E-9 &$3.96$\\
\hline
\end{tabular}
\end{table}

\begin{table}[!h]
\centering
\small{\caption{\emph{Errors and convergence orders for the RSV flux function approximation for Example 2.}}}
\label{Tab:Eg2b}
\begin{tabular}{lllllllllllll}
\hline
 $~~~~k$&$~~n $ &~~$e_{f}$   &order &~~$e_{f,c}$  &order  &~~$e_{f,r}$&order  &~~$e_{f,n}$  &order &~~$e_{f,l}$  &order\\
\hline
\multirow{4}{*}{$k=1$}
 &$~128$  & 4.97E-6 &      &4.81E-6 &     &8.74E-6 &      &7.21E-6 &         &2.18E-4 &   \\
 &$~256$  & 6.24E-7 &3.00  &6.03E-7 &3.00 &1.10E-6 &$3.00$&9.06E-7 &$~3.00$  &5.36E-5 &$~2.02$   \\
 &$~512$  & 7.80E-8 &3.00  &7.55E-8 &3.00 &1.38E-7 &$3.00$&1.13E-7 &$~3.00$  &1.34E-5 &$~2.00$   \\
 &$1024$  & 9.75E-9 &3.00  &9.44E-9 &3.00 &1.71E-8 &$3.00$&1.42E-8 &$~3.00$  &3.35E-6 &$~2.00$   \\
\hline
\multirow{4}{*}{$k=2$}
 &$~128$  & 2.85E-9  &      &2.86E-9  &      &3.31E-8  &       &1.29E-8  &         &2.97E-6 &   \\
 &$~256$  & 8.91E-11 &5.00  &8.52E-11 &5.07  &2.07E-9  &$4.00$&6.61E-10  &$~4.30$  &3.74E-7 &$~3.00$   \\
 &$~512$  & 2.78E-12 &5.00  &2.66E-12 &5.00  &1.29E-10 &$4.00$&3.33E-11  &$~4.31$  &4.68E-8 &$~3.00$   \\
 &$1024$  & 8.70E-14 &5.00  &8.31E-14 &5.00  &8.06E-11 &$4.00$&1.69E-12  &$~4.30$  &5.85E-9 &$~3.00$   \\
\hline
\multirow{4}{*}{$k=3$}
 &$~~32$  & 4.71E-8 &       &4.70E-8  &      &2.77E-7  &      &9.46E-8  &         &9.03E-6 &   \\
 &$~~64$  & 8.39E-10 &5.81  &8.40E-10 &5.81  &8.67E-9  &$5.00$&2.64E-9  &$~5.16$  &5.71E-7 &$~3.98$   \\
 &$~128$  & 2.25E-11 &5.22  &2.25E-11 &5.22  &2.55E-10 &$5.08$&6.60E-11 &$~5.32$  &3.44E-8 &$~4.05$   \\
 &$~256$  & 5.91E-13 &5.25  &6.01E-13 &5.25  &7.97E-12 &$5.01$&1.68E-12 &$~5.30$  &2.15E-9 &$~4.00$   \\
\hline
\end{tabular}
\end{table}

\begin{table}[!h]
\centering
\small{\caption{\emph{Errors and   convergence orders for the RSV solution approximation Example 2.}}}
\label{Tab:Eg2c}
\begin{tabular}{lllllllllllll}
\hline
 $~~~~k$&$~~n $ &~~$e_{u}$   &order &~~$e_{u,c}$  &order  &~~$e_{u,r}$&order  &~~$e_{u,n}$  &order &~~$e_{u,l}$  &order\\
\hline
\multirow{4}{*}{$k=1$}
 &$~128$  & 1.58E-6 &      &1.62E-5 &       &5.25E-5 &      &1.51E-4 &         &6.17E-3 &   \\
 &$~256$  & 2.16E-7 &2.87  &2.20E-6 &$2.88$ &1.08E-5 &$2.28$&3.19E-5 &$~2.25$  &2.61E-3 &$~1.24$   \\
 &$~512$  & 2.98E-8 &2.86  &3.03E-7 &$2.86$ &2.26E-6 &$2.26$&6.72E-6 &$~2.25$  &1.10E-3 &$~1.24$   \\
 &$1024$  & 4.13E-9 &2.85  &4.22E-8 &$2.84$ &4.73E-7 &$2.26$&1.41E-6 &$~2.25$  &4.62E-4 &$~1.25$   \\
\hline
\multirow{4}{*}{$k=2$}
 &$~128$  & 7.97E-8  &      &7.96E-8 &       &6.20E-7 &      &8.37E-7 &         &8.46E-5 &   \\
 &$~256$  & 6.26E-9  &3.67  &6.18E-9 &$3.67$ &6.73E-8 &$3.20$&8.90E-8 &$~3.23$  &1.83E-5 &$~2.21$   \\
 &$~512$  & 4.75E-10 &3.72  &4.71E-10&$3.71$ &7.22E-9 &$3.22$&9.41E-9 &$~3.24$  &3.91E-6 &$~2.23$   \\
 &$1024$  & 3.58E-11 &3.73  &3.56E-11&$3.72$ &7.68E-10 &$3.23$&9.92E-10&$~3.25$ &8.28E-7 &$~2.24$   \\
\hline
\multirow{4}{*}{$k=3$}
 &$~~32$  & 2.50E-7 &       &2.50E-7 &        &4.81E-6 &      &3.87E-6 &          &1.75E-4 &   \\
 &$~~64$  & 1.27E-8 &4.31  &1.27E-8 &$4.31$  &2.33E-7 &$4.37$&1.78E-7 &$~4.44$  &1.69E-5 &$~3.36$   \\
 &$~128$  & 5.42E-10 &4.54  &5.43E-10 &$4.54$  &1.26E-9 &$4.21$&9.83E-9 &$~4.18$  &1.80E-6 &$~3.23$   \\
 &$~256$  & 2.37E-11 &4.52  &2.37E-11 &$4.52$ &6.68E-10 &$4.23$&5.27E-10 &$~4.22$ &1.93E-7 &$~3.22$   \\
\hline
\end{tabular}
\end{table}

\begin{table}[!h]
\centering
\small{\caption{\emph{Errors and  convergence orders for the LSV flux
function approximation Example 2.}}}
\label{Tab:Eg2d}
\begin{tabular}{lllllllllllll}
\hline
 $~~~~k$&$~~n $ &~~$e_f$   &order &~~$e_{f,c}$  &order  &~~$e_{f,r}$&order  &~~$e_{f,n}$  &order &~~$e_{f,l}$  &order\\
\hline
\multirow{4}{*}{$k=1$}
 &$~128$  & 1.34E-4 &      &1.24E-4 &       &1.87E-4 &      &1.31E-4 &         &7.74E-4 &   \\
 &$~256$  & 3.30E-5 &2.01  &3.08E-5 &$2.01$ &4.63E-5 &$2.02$&3.26E-5 &$~2.01$  &1.93E-4 &$~2.00$   \\
 &$~512$  & 8.25E-6 &2.00  &7.70E-6 &$2.01$ &1.15E-5 &$2.01$&8.14E-6 &$~2.01$  &4.78E-3 &$~2.00$   \\
 &$1024$  & 2.06E-6 &2.00  &1.92E-6 &$2.01$ &2.87E-6 &$2.01$&2.03E-6 &$~2.01$  &1.19E-3 &$~2.00$   \\
\hline
\multirow{4}{*}{$k=2$}
 &$~128$  & 3.36E-8 &      &3.46E-8 &       &8.22E-8 &      &5.03E-8 &         &3.21E-6 &   \\
 &$~256$  & 2.27E-9 &4.00  &2.16E-9 &$4.00$ &5.18E-9 &$3.99$&3.15E-9 &$~4.00$  &3.97E-7 &$~3.01$   \\
 &$~512$  & 1.42E-10 &4.00  &1.35E-10 &$4.00$ &3.27E-10 &$3.99$&1.98E-10 &$~4.00$  &4.94E-8 &$~3.01$   \\
 &$1024$  & 8.86E-12 &4.00  &8.43E-12 &$4.00$ &2.04E-11 &$4.00$&1.24E-11 &$~4.00$  &6.17E-9 &$~3.00$   \\
\hline
\multirow{4}{*}{$k=3$}
 &$~~32$  & 1.29E-7  &      &1.29E-7  &      &5.07E-7  &      &2.02E-7  &         &1.68E-5 &   \\
 &$~~64$  & 5.17E-9  &4.64  &5.17E-9  &4.64  &1.84E-8  &$4.79$&7.32E-9  &$~4.79$  &7.54E-7 &$~3.96$   \\
 &$~128$  & 1.95E-10 &4.73  &1.95E-10 &4.73  &6.00E-10 &$4.94$&2.31E-10 &$~4.98$  &4.66E-8 &$~4.01$   \\
 &$~256$  & 5.10E-12 &5.25  &5.10E-12 &5.25  &1.88E-11 &$5.00$&7.22E-12 &$~5.00$  &2.91E-9 &$~4.00$   \\
\hline
\end{tabular}
\end{table}

\begin{table}[!h]
\centering
\small{\caption{\emph{Errors and   convergence rates for the LSV solution approximation for Example 2.}}}
\label{Tab:Eg2e}
\begin{tabular}{lllllllllllll}
\hline
 $~~~~k$&$~~n $ &~~$e_u$   &order &~~$e_{u,c}$  &order  &~~$e_{u,r}$&order  &~~$e_{u,n}$  &order &~~$e_{u,l}$  &order\\
\hline
\multirow{4}{*}{$k=1$}
 &$~128$  & 3.26E-4 &      &2.80E-4 &       &3.25E-4 &      &3.41E-4 &         &7.51E-3 &   \\
 &$~256$  & 8.05E-5 &2.01  &6.92E-5 &$2.02$ &8.05E-5 &$2.02$&8.31E-5 &$~2.04$  &3.07E-3 &$~1.29$   \\
 &$~512$  & 2.01E-5 &2.00  &1.72E-5 &$2.01$ &2.00E-5 &$2.01$&2.04E-5 &$~2.02$  &1.27E-3 &$~1.27$   \\
 &$1024$  & 5.03E-6 &2.00  &4.28E-6 &$2.01$ &4.98E-6 &$2.01$&5.05E-6 &$~2.02$  &5.34E-4 &$~1.25$   \\
\hline
\multirow{4}{*}{$k=2$}
 &$~128$  & 1.19E-7 &      &1.21E-7 &       &6.16E-7 &      &6.06E-7 &         &8.38E-5 &   \\
 &$~256$  & 8.17E-9 &3.86  &8.28E-9 &$3.87$ &6.66E-8 &$3.21$&6.28E-8 &$~3.27$  &1.80E-5 &$~2.22$   \\
 &$~512$  & 5.63E-10&3.86  &5.77E-10 &$3.84$ &7.16E-9 &$3.22$&6.58E-9 &$~3.26$  &3.84E-6 &$~2.23$   \\
 &$1024$  & 3.90E-11&3.85  &4.09E-11 &$3.82$ &7.65E-10 &$3.23$&6.91E-10 &$~3.25$  &8.13E-7 &$~2.24$   \\
\hline
\multirow{4}{*}{$k=3$}
 &$~~32$  & 3.01E-7 &      &3.00E-7 &       &3.26E-6 &      &2.14E-6 &         &1.68E-5 &   \\
 &$~~64$  & 1.77E-8 &4.08  &1.77E-8 &$4.08$ &1.85E-7 &$4.14$&9.17E-8 &$~4.54$  &7.54E-7 &$~3.31$   \\
 &$~128$  & 7.31E-10 &4.59  &7.31E-10 &$4.60$ &9.81E-9 &$4.24$&5.25E-9 &$~4.13$  &4.66E-8 &$~3.22$   \\
 &$~256$  & 3.03E-11 &4.60  &3.04E-11 &$4.60$ &5.15E-10 &$4.25$&2.80E-10 &$~4.23$  &4.97E-9 &$~3.23$   \\
\hline
\end{tabular}
\end{table}

\begin{table}[!h]
	\small{\caption{\emph{Example 1-- Comparison  between  SV methods  and upwind DG method in the cases $k=1,2,3.$}}}
	\label{Tab:Eg12a}	\centering
	\begin{tabular}{ll|llllll|llllll}
		\hline
		$~ $  &       &\multicolumn{6}{c}{RSV}& \multicolumn{6}{c}{LSV} \\
		\hline
		$~~k$&$~~n $ &~~$\|\bar{e}\|_{0}$   &~~r  &~~$\|\bar{e}\|_{f,c}$  &~~r  &~~$\|\bar{e}\|_{f,c}$  &~~r &~~$\|e\|_0$   &~~r  &~$\|\bar{e}\|_{f,c}$  &~~r &~$\|\bar{e}\|_{u,c}$  &~~r\\
		\hline
		\multirow{4}{*}{$k=1$}
		& $~128$ & 5.1E-05 &      &5.8E-06 &       &5.8E-06 &       & 6.3E-04 &      &5.6E-05 &      &5.4E-05 &      \\
		& $~256$ & 9.0E-06 &2.5  &7.4E-07 &$3.0$ &7.4E-07 &$3.0$ & 1.6E-04 &2.0  &1.4E-05 &$2.0$     &1.4E-05 &$2.0$\\
		& $~512$ & 1.6E-06 &2.5  &9.3E-08 &$3.0$ &9.3E-08 &$3.0$ & 4.0E-05 &2.0  &3.6E-06 &$2.0$     &3.4E-06 &$2.0$\\
		& $1024$ & 2.8E-07 &2.5  &1.2E-08 &$3.0$ &1.2E-08 &$3.0$ & 9.9E-06 &2.0  &9.0E-06 &$2.0$     &8.4E-06 &$2.0$ \\
		\hline
		\multirow{4}{*}{$k=2$}
		& $~128$ & 4.5E-08 &      &1.9E-08 &     &1.9E-08 &       &4.4E-06 &      &2.5E-08 &       &3.4E-08 &       \\
		& $~256$ & 3.6E-09 &3.6  &1.1E-09 &$4.1$ &1.1E-09 &$4.1$ &5.5E-07 &3.0  &1.5E-09 &$4.0$    &1.5E-09 &$4.0$\\
		& $~512$ & 3.1E-10 &3.5  &6.8E-11 &$4.0$ &6.8E-11 &$4.0$ &6.9E-08 &3.0  &9.6E-11 &$4.0$    &9.4E-11 &$4.0$\\
		& $1024$ & 2.7E-11 &3.5  &4.3E-12 &$4.0$ &4.3E-12 &$4.0$ &8.6E-09 &3.0  &6.0E-11 &$4.0$    &5.9E-12 &$4.0$\\
		\hline
		\multirow{4}{*}{$k=3$}
		& $~~32$ & 1.6E-06 &     &4.1E-07 &      &4.1E-07 &      &7.7E-06 &      &1.3E-07 &      &1.2E-07 &\\
		& $~~64$ & 6.8E-08 &4.6  &9.7E-09 &$5.4$ &9.7E-09 &$5.4$ &4.8E-07 &4.0   &4.2E-09 &$4.9$ &4.0E-09 &$4.9$\\
		& $~128$ & 3.0E-09 &4.5  &2.4E-10 &$5.3$ &2.4E-10 &$5.3$ &3.0E-08 &4.0   &1.4E-11 &$5.0$ &1.3E-11 &$5.0$\\
		& $~256$ & 1.3E-10 &4.5  &6.4E-12 &$5.3$ &6.4E-12 &$5.3$ &1.9E-09 &4.0   &4.2E-12 &$5.0$ &3.9E-12 &$5.0$\\
		\hline
	\end{tabular}
\end{table}

\begin{table}[!h]
	\small{\caption{\emph{Example 2-- Comparison results of SV and upwind DG in the cases $k=1,2,3.$}}}
	\label{Tab:Eg12b}	\centering
	\begin{tabular}{ll|llllll|llllll}
		\hline
		$~ $  &       &\multicolumn{6}{c}{RSV}& \multicolumn{6}{c}{LSV} \\
		\hline
		$~~k$&$~~n $ &~~$\|\bar{e}\|_0$   &r  &~~$\|\bar{e}\|_{f,c}$  &r&~~$\|\bar{e}\|_{u,c}$  &r  &~~$\|\bar{e}\|_0$   &r &~~$\|\bar{e}\|_{f,c}$  &r &~~$\|\bar{e}\|_{u,c}$  &r\\
		\hline
		\multirow{4}{*}{$k=1$}
		& $~128$ & 5.2E-05 &     &9.6E-07 &      &9.4E-07 &      & 6.3E-04 &     &5.1E-05  &     &5.0E-05  & \\
		& $~256$ & 9.1E-06 &2.5  &1.2E-07 &$3.0$ &1.2E-07 &$3.0$ & 1.6E-04 &2.0  &1.3E-05 &$2.0$ &1.2E-05 &$2.0$\\
		& $~512$ & 1.6E-06 &2.5  &1.5E-08 &$3.0$ &1.4E-08 &$3.0$ & 4.0E-05 &2.0  &3.3E-06 &$2.0$ &3.2E-06 &$2.0$\\
		& $1024$ & 2.9E-07 &2.5  &1.9E-09 &$3.0$ &1.9E-09 &$3.0$ & 9.9E-06 &2.0  &8.2E-07 &$2.0$ &8.1E-07 &$2.0$\\
		\hline
		\multirow{4}{*}{$k=2$}
		& $~128$ & 6.5E-08 &     &4.5E-10 &      &4.2E-10 &       &4.4E-06 &     &2.2E-08 &      &2.0E-08 &      \\
		& $~256$ & 4.9E-09 &3.7  &1.5E-11 &$4.9$ &1.4E-11 &$4.9$  &5.5E-07 &3.0  &1.4E-09 &$4.0$ &1.3E-09 &$4.0$\\
		& $~512$ & 4.9E-10 &3.3  &4.9E-13 &$4.9$ &4.4E-13 &$5.0$  &6.9E-08 &3.0  &8.6E-11 &$4.0$ &8.5E-11 &$4.0$ \\
		& $1024$ & 4.3E-11 &3.5  &1.5E-14 &$5.0$ &1.4E-14 &$5.0$  &8.6E-09 &3.0  &5.4E-12 &$4.0$ &5.3E-12 &$4.0$\\
		\hline
		\multirow{4}{*}{$k=3$}
		& $~~32$ & 1.3E-06 &     &1.3E-08 &       &1.2E-08 &      &7.7E-06 &     &7.3E-08 &      &7.0E-08 &      \\
		& $~~64$ & 5.3E-08 &4.6  &3.0E-10 &$5.4$  &2.9E-10 &$5.4$ &4.8E-07 &4.0  &3.2E-09 &$4.5$ &3.0E-09 &$4.5$\\
		& $~128$ & 2.2E-09 &4.6  &7.1E-12 &$5.4$  &6.9E-12 &$5.4$ &3.0E-08 &4.0  &1.0E-10 &$5.0$ &9.8E-11 &$5.0$\\
		& $~256$ & 9.8E-11 &4.5  &1.5E-13 &$5.5$  &1.3E-11 &$5.5$ &1.9E-09 &4.0  &3.1E-12 &$5.0$ &2.9E-12 &$5.0$\\
		\hline
	\end{tabular}
\end{table}

To compare the differences  between our SV methods and the upwind DG methods,
we also provide comparisons between
the two methods in terms of the error and accuracy.    Define
\[
   \bar{e}=u_h-\bar u_h,
\]
where  $\bar u_h$ is the solution of the upwind DG method. We list in Tables 6.11 and 6.12  variou errors of $\bar e$ in $L^2$ norm and  for cell average for  Example 1 and Example 2, respectively.

We observe that the $\|\bar{e}\|_{0}$ of RSV and LSV converges with the rates of $k+\frac{3}{2}$ and $k+1$, which indicates that the RSV solution is superclose to the DG solution.
We also observe that the convergence rates of $\bar{e}_{f,c}$ and $\bar{e}_{u,c}$ are at least $k+2$ for both LSV and RSV schemes.
When $k=1$, there exists no superconvergence phenomenon for the LSV method. These results are consistent with our theoretical findings.

\section{Concluding remarks}
In this work, we study the $L^2$-norm stability, convergence and superconvergence behaviors of LSV and RSV for 1-D linear hyperbolic equations with degenerate variable coefficients.
We prove that both SV methods are stable and have optimal convergence orders in the $L^2$ space.
Furthermore, we demonstrate the superconvergence behaviors including: the flux function approximation $\alpha u_h$ are of $2k$-th order superconvergent towards the flux function $\alpha u_I$, of $(k+2)$-th order superconvergent at the interpolation points and at downwind points, of $(k+2)$-th order superconvergent for the cell average, and the derivative of the flux function approximation $\alpha \partial_x u_h$  is of $(k+1)$-th order superconvergent at roots of  $\partial_x (\prod_{j=1}^{k+1}(x-y_{i,j}))$;
the convergence rate of the SV approximation solution itself $u_h$ depends upon the the specific property of $\alpha$, revealing that the highest superconvergence rate that can be achieved is $k+\frac{3}{2}$, which is half an order higher than the optimal convergence rate. These superconvergent results are the same as those obtained by upwind DG method. The proposed schemes can be easily extended to 2-D case. Numerical examples are provided to verify the theoretical findings.

The mathematical theory of the SV is far from developed.
More works are called for to find a better and more intrinsic understanding of the SV.
More works including the theory for higher dimensional conservation laws and nonlinear problems are under way.
%

\bigskip
\bigskip
\end{document}